\documentclass[11pt,a4paper]{article}
\usepackage{amsfonts}
\usepackage{amssymb}
\usepackage{mathrsfs}
\usepackage{amsmath}
\usepackage{amsthm}
\usepackage{enumerate}
\usepackage{indentfirst}
\usepackage{amsfonts}
\usepackage{amssymb}
\usepackage{mathrsfs}
\usepackage{amsmath}
\usepackage{amsthm}
\usepackage{enumerate}
\usepackage{mathrsfs}
\usepackage{pxfonts}
\usepackage{geometry}

\usepackage{authblk}

\usepackage{cite}

\usepackage{enumitem}  
\usepackage{calc}
\setlist{labelindent=1pt,itemsep=.5em}
\setlist[itemize]{leftmargin=1.2cm}
\setlist[enumerate]{itemindent=0em,leftmargin=1.2cm}
\setlist[enumerate,1]{label={\upshape(\roman*)}}

\allowdisplaybreaks
\newtheorem{defn}{Definition}[section]
\newtheorem{thm}[defn]{Theorem}
\newtheorem{lem}[defn]{Lemma}
\newtheorem{prop}[defn]{Proposition}
\newtheorem{cor}[defn]{Corollary}
\newtheorem{ex}[defn]{Example}
\newtheorem{re}[defn]{Remark}

\def\c{\cdot}

\def\o{\otimes}

\numberwithin{equation}{section}

\newcommand{\AKMSbracket}[1]{\left[#1\right]}
\newcommand{\AKMSpara}[1]{\left(#1\right)}

\makeatletter
\newcommand{\subjclass}[2][2010]{%
  \let\@oldtitle\@title%
  \gdef\@title{\@oldtitle\footnotetext{#1 \emph{Mathematics subject classification}: #2}}%
}
\newcommand{\keywords}[1]{%
  \let\@@oldtitle\@title%
  \gdef\@title{\@@oldtitle\footnotetext{\emph{Keywords}: #1.}}%
}
\makeatother

\title{Structure and cohomology of 3-Lie-Rinehart superalgebras}
\author{Abdelkader Ben Hassine$^{1,2}$, Taoufik Chtioui$^{2}$, \authorcr Sami Mabrouk$^{3}$, Sergei Silvestrov$^{4}$\footnote{ Corresponding author: sergei.silvestrov@mdh.se} \\
\small{$^{1}$Department of Mathematics, Faculty of Science and Arts at Belqarn, \\
P. O. Box 60, Sabt Al-Alaya 61985, University of Bisha, Kingdom of Saudi Arabia \authorcr
$^{2}$Faculty of Sciences, University of Sfax, BP 1171, 3000 Sfax, Tunisia \authorcr
$^{3}$Faculty of Sciences, University of Gafsa, BP 2100, Gafsa, Tunisia \authorcr
$^{4}$Division of Applied Mathematics, School of Education, Culture and Communication, M\"{a}lardalen University, Box 883, 72123 V\"{a}steras, Sweden. }}


\date{}

\begin{document}

\maketitle

\abstract{We introduce the concept of 3-Lie-Rinehart superalgebra and systematically describe a cohomology complex by considering coefficient modules. Furthermore, we study the relationships between a Lie-Rinehart superalgebra and  its induced 3-Lie-Rinehart superalgebra. The deformations of 3-Lie-Rinehart superalgebra are considered via the cohomology theory.}

\section*{Introduction}
The notion of Lie-Rinehart algebras was introduced by J. Herz in \cite{Herz} and mainly
developed in \cite{Palais,Rinehart2}. A Lie-Rinehart algebra can be thought as a Lie $\mathbb{K}$-algebra, which is simultaneously an $A$-module, where $A$ is an associative and
commutative $\mathbb{K}$-algebra, in such a way that both structures are related by a compatibility condition. A first approach to this class of algebras can be found in \cite{Huebschmann:Poissoncohomquantiz,Huebschmann1}.

G. Rinehart developed in \cite{Rinehart2} a formalism of differential forms for general commutative algebras which relies on the notion of $(\mathbb{K}, A)$-Lie algebra where $\mathbb{K}$ is a commutative ring with unit and $A$ a commutative $\mathbb{K}$-algebra. The notion of Lie-Rinehart algebra includes
an abstract algebraic characterization of the algebraic structure which underlies a Lie algebroid \cite{HigginsMackenzie,Huebschmann:Poissoncohomquantiz,Mackenzie,Mackenzie1}. Thus a Lie-Rinehart
algebra is an algebraic generalization of the notion of a Lie algebroid: the space of sections of a vector bundle is replaced by a
module over a ring, a vector field by a derivation of the ring. For further details and a history of the notion of
Lie-Rinehart algebra, we refer to \cite{Huebschmann3}.
Lie-Rinehart algebras have been investigated further in many papers \cite{Bkouche,Casas,ChenLiuZhong,Huebschmann:Poissoncohomquantiz,Huebschmann1,Huebschmann:LieRinGerstenhBatalinVil,Huebschmann:ExtensionsLieRinehart,Huebschmann:MultiderivMaurerCartanSHLieRin,KrahmerRovi,Mackenzie}.

Some generalizations of Lie-Rinehart algebras, such as Lie-Rinehart superalgebras \cite{Chemla} or restricted Lie-Rinehart algebras \cite{Dokas},
have been recently studied.  The cohomology $H^*_{Rin}(L; M)$ of a Lie-Rinehart algebra $L$ with coefficients in a
Lie-Rinehart module $M$ was first defined by Rinehart \cite{Rinehart2} and further developed in \cite{Huebschmann:Poissoncohomquantiz}.

In \cite{CasasLadraPirashvili,CasasLadraPirashvili1,Casas},
the authors introduced the notion of cross modules of Lie-Rinehart algebras. Hom-Lie-Rinehart algebras and their extensions in the small dimension cohomology space was introduced and studied in \cite{MandalMishra:HomLieRinehartalgebras}.
Hom-Rinehart algebras have close relations with Hom-Gerstenhaber Algebras and Hom-Lie Algebroids \cite{MandalMishra:HomGerstenhaberHomLiealgebroids,MishraSilvestrov:SpringerAAS2020HomGerstenhalgsHomLiealgds}.

The study of $3$-Lie algebras \cite{Filippov} gets a lot of attention since it has close relationships with Lie algebras, Hom-Lie algebras, commutative associative algebras, and cubic matrices \cite{BaiBaiWang,BaiWu,BaiLiWang,CalderonPiulestan}. For example, it is applied to the study of Nambu mechanics and the study of supersymmetry and gauge symmetry transformations of the world-volume theory of multiple coincident M2-branes \cite{AlbuquerqueBarreiroCalderonAnchez,MandalMishra:HomLieRinehartalgebras,Rota,Takhtajan}.

In 1996, the concept of $n$-Lie superalgebras was firstly
introduced by Daletskii and Kushnirevich in \cite{Yu&Daletskii&Kushnirevich}. Moreover, Cantarini and Kac gave a more general concept of
$n$-Lie superalgebras again in 2010 (see \cite{CantariniKac}). $n$-Lie superalgebras
are more general structures including $n$-Lie algebras, $n$-ary Nambu-Lie superalgebras, and Lie superalgebras.
The construction of $(n+1)$-Lie algebras induced by $n$-Lie algebras using combination of bracket multiplication with a trace, motivated by the works on the quantization of the Nambu brackets \cite{almy:quantnambu}, was generalized using the brackets of general Hom-Lie algebra or $n$-Hom-Lie algebra and trace-like linear forms satisfying some conditions depending on the linear maps defining the Hom-Lie or $n$-Hom-Lie algebras in \cite{ams:ternary,ams:n}. The structure of $3$-Lie algebras induced by Lie algebras, classification of $3$-Lie algebras and application to constructions of B.R.S. algebras have been considered in \cite{Abramov2018:WeilAlg3LiealgBRSalg,AbramovLatt2016:classifLowdim3Liesuperalg,Abramov2017:Super3LiealgebrasinducedsuperLiealg}.
Interesting constructions of ternary Lie superalgebras in connection to superspace extension of Nambu-Hamilton equation is considered in \cite{AbramovLatt:SpringerAAS2020TernLieSuperalgsNambuHamiltoneqs}. In \cite{HMN}, a method was demonstrated of how to construct $n$-ary multiplications from the binary multiplication of a Hom-Lie algebra and a $(n-2)$-linear function satisfying certain compatibility conditions. Solvability and Nilpotency for $n$-Hom-Lie Algebras and
$(n+1)$-Hom-Lie Algebras induced by $n$-Hom-Lie Algebras have been considered in \cite{kms:solvnilpnhomlie2020}.
In \cite{Abramov2017:Super3LiealgebrasinducedsuperLiealg,Abramov2019:3LiesuperalgsLiesuperals,AbramovLatt:SpringerAAS2020TernLieSuperalgsNambuHamiltoneqs}, $3$-Lie superalgebras were constructed starting with a Lie superalgebras and various properties of such algebras considered. Related constructions for $n$-ary hom-Lie algebras and for $n$-ary hom-Lie superalgebras can be found in \cite{AbramovSilvestrov:3homLiealgsigmaderivINvol,GuanChenSun,HMN,KitouniMakhloufSilvestrov,akms:ternary,ams:ternary,ams:n}.
The ternary case of (Hom-)Lie Rinehart algebras was developed  in \cite{BaiLiWu,GuoZhangWang,BaiLieWu1}.

This paper is organized as follows.  In Section \ref{Sec:DefNotations}, we recall the notion of Lie-Rinehart superalgebra and define their cohomology group. We generalize the Lie-Rinehart superalgebra to the ternary case. Section \ref{sec:rbl3rbl} is devoted to some construction results. We begin by constructing $3$-Lie-Rinehart superalgebras starting with a Lie-Rinehart superalgebras.  We construct some new $3$-Lie-Rinehart superalgebras from a given $3$-Lie-Rinehart superalgebra.  The notion of a module for a $3$-Lie-Rinehart superalgebras appeared in Section \ref{sec:cohomdef3LieRinesuperalg}, and subsequently we introduce a cochain complex and cohomology of a $3$-Lie-Rinehart superalgebras with coefficients in a module. Then we study relations between 1 and 2 cocycles of a Lie-Rinehart superalgebra and the induced $3$-Lie-Rinehart superalgebra. At the end of this section, we study the deformation of $3$-Lie-Rinehart superalgebras.

\section{Definitions and Notations} \label{Sec:DefNotations}
In this section, we review basic definitions of Lie superalgebras, $3$-Lie superalgebras, Lie-Rinehart superalgebras and generalize the notion of $3$-Lie-Rinhehart algebras to the super case.

Let $V= V_{\overline 0}\oplus V_{\overline 1}$ be a $\mathbb{Z}_2$-graded vector space. If
$v \in V$ is a homogenous element, then its degree will be denoted by $\bar v$,
where $\bar v\in \mathbb{Z}_2$ and $\mathbb{Z}_2=\{\overline 0,\overline 1\}$. Let $End(V)$ be the $\mathbb{Z}_2$-graded vector space of
endomorphisms of a $\mathbb{Z}_2$-graded vector space $V= V_{\overline 0}\oplus V_{\overline 1}$. Denoted by $\mathcal{H}(V)$ the set of homogenous elements of $V$. The composition of
two endomorphisms $a\circ b$ determines the structure of Lie superalgebra in $End(V)$
and the graded binary commutator $[a, b] = a\circ b - (-1)^{\bar a\bar b}b \circ a$ induces the
structure of Lie superalgebra in $End(V)$.

\subsection{Lie-Rinehart superalgebras}

\begin{defn}[\cite{Kac}]
A Lie superalgebra is a pair $(\mathfrak g, [\cdot, \cdot])$ consisting
of a $\mathbb{Z}_2$-graded vector space $\mathfrak g=\mathfrak g_{\overline 0}\oplus \mathfrak g_{\overline 1}$, an even bilinear map $[\cdot,\cdot] : \mathfrak g\times\mathfrak g\longrightarrow\mathfrak g$  satisfying the following
 identities:
\begin{eqnarray}
& &[x,y] = -(-1)^{\bar x\bar y}[y, x],\ \text{super skew-symetric}\label{skewsupersymmetry}\\
& &\displaystyle\circlearrowleft_{x,y,z}(-1)^{\bar x\bar z}[x, [y,z]] = 0,\  \text{super-Jacobi}\label{SuperJacobi}
\end{eqnarray}
where $x, y$ and $z$ are homogeneous elements in $\mathfrak g$.
\end{defn}

\begin{defn}
A representation of a Lie superalgebra $(\mathfrak g, [\cdot,\cdot])$
on a $\mathbb Z_2$-graded vector space $V=V_{\overline 0} \oplus V_{\overline 1}$ is an
even linear map $\mu:\mathfrak g\rightarrow\mathfrak{gl}(V)$, such that for all homogeneous $x, y \in \mathfrak g,$ the following
identity is satisfied:
\begin{equation}\label{rep-super}
\mu([x, y])= \mu(x)\mu(y)-(-1)^{\bar x\bar y}\mu(y)\mu(x).
\end{equation}
\end{defn}
Now, we recall the definition of Lie-Rinehart superalgebra \cite{Chemla}.
\begin{defn} \label{rinehart-super}
A Lie-Rinehart superalgebra $L$ over (an associative supercommutative superalgebra) $A$ is a Lie superalgebra
over $\mathbb K$ with an $A$-module structure and an even linear map $\mu: L \rightarrow Der(A)$, such that
 that following conditions hold:
 \begin{enumerate}
\item $\mu$ is a representation of $(L, [\cdot,\cdot])$ on $A$.
\item $\mu(ax) = a\mu(x)$ for all $a\in A, x \in L$.
\item The compatibility condition:
\begin{equation}
\forall a\in \mathcal H(A), x,y \in \mathcal H(L):
\quad [x, ay] = \mu(x)ay + (-1)^{\bar a\bar x} a[x, y].
\label{compatibility-rinhe-super}
\end{equation}
\end{enumerate}
\end{defn}
\begin{ex}
Let us observe that Lie-Rinehart superalgebras over $A$ with trivial  map
$\mu : L\to  Der(A)$ are exactly Lie superalgebras. If $A = \mathbb{K}$, then $Der(A)=0$ and there is no
difference between Lie and Lie-Rinehart superalgebras. Therefore the concept of Lie-Rinehart
superalgebras generalizes the concept of Lie superalgebras.
\end{ex}
\begin{ex}
Let $A$ an associative supercommutative superalgebra. Then, the $A$-module $Der(A)\oplus A$ is a Lie-Rinehart superalgebra over $A$ with the
bracket
$$
[(D, a), (D', a')] = ([D, D'],  D(a') -(-1)^{\bar D'\bar a} D'(a)),
$$
and an even map $p_1 : Der(A) \oplus A \to Der(A)$, the projection onto the first factor.
\end{ex}
Now, we define the cohomology of a Lie-Rinehart superalgebra. First we introduce the notion of left module
over a Lie-Rinehart superalgebra.

\begin{defn}
Let $M$ be an $A$-module. Then  $M$ is a left module over a
Lie-Rinehart superalgebra $L$  if there exits  an even map $\theta : L\otimes M \rightarrow M$  such that:
\begin{enumerate}
 \item $\theta$ is a representation of the Lie superalgebra
$(L, [\cdot,\cdot])$ on $M$.
\item $\theta(ax, m) = a\theta (x, m)$ for all $a \in A, x \in  L, m \in  M.$
\item $\theta(x, am) = (-1)^{\bar a\bar x}a\theta(x, m)+ \rho(x)(a)(m)$ for all $x\in \mathcal H(L),\ a \in \mathcal H(A), m \in M.$
\end{enumerate}
\end{defn}

Let $(L, A,[\cdot,\cdot],\mu)$ be a Lie-Rinehart superalgebra, and let $M$ be a left  module over $L$. For Lie-Rinehart superalgebra $L$ with coefficients in $M$, consider the $\mathbb{Z}_+$-graded space of $\mathbb{K}$-modules
$$C^*(L; M) := \displaystyle\oplus_{n\geq0}C^n(L; M)$$ where $C^n(L; M) \subset Hom(\wedge^nL, M)$
consisting of elements satisfying:
  $$f(x_1,\dots,ax_i,\dots, x_n) =(-1)^{\bar a(\bar x_1+\bar x_2+\dots+\bar x_{i-1}+\bar f)}af(x_1, \dots, x_i,\dots, x_n),$$
   for all $x_i \in\mathcal H(L), i=1,\dots, n$ and $a \in\mathcal H(A)$.

Define the even linear map $\delta_{LR}: C^n(L; M)\rightarrow C^{n+1}(L; M)$ given by
\begin{multline*}
\delta_{LR} f(x_1,\dots, x_{n+1}) :=\displaystyle\sum^{n+1}_{i=1}(-1)^{i+1+\bar x_i(\bar f+\bar x_1+\dots+\bar x_{i-1})}\theta(x_i,f(x_1,\dots, \widehat{x}_i,\dots, x_{n+1}))\\
+\displaystyle\sum_{1\leq i<j\leq n+1}(-1)^{(\bar x_i+\bar x_j)(\bar x_1+\dots+\bar x_{i-1}+\bar x_{i+1}+\dots+\bar x_{j-1})}f([x_i, x_j], x_1,\dots,\widehat{x}_i, \dots,\widehat{x}_j, \dots,x_{n+1})
\end{multline*}
for all $f \in C^n(L; M),\ x_i \in \mathcal H(L),$ where $1 \leq i \leq n+1$.

With the above notation, the map $\delta_{LR}$ gives rise to a coboundary map.
\begin{prop}
If $f \in C^n(L; M)$, then $\delta^2_{LR}f \in C^{n+1}(L; M)$ and $\delta^2_{LR}=0.$
\end{prop}

By the above proposition, $(C^*(L, M),\delta_{LR})$ is a cochain complex. The resulting cohomology of the
cochain complex we define to be the cohomology space of Lie-Rinehart superalgebra $(L, A,[\cdot,\cdot],\mu)$ with
coefficients in $M$, and we denote this cohomology as $H^*(L, M).$

\subsection{$3$-Lie-Rinehart superalgebras}

\begin{defn}[\cite{CantariniKac}]
A $\mathbb{Z}_2$-graded vector space $\mathfrak{g}=\mathfrak g_{\overline 0}\oplus\mathfrak g_{\overline 1}$ is said to be a $3$-Lie superalgebra, if it is endowed with an even trilinear map (bracket) $[\cdot,\cdot,\cdot]:\mathfrak g\times \mathfrak g \times \mathfrak g\rightarrow \mathfrak g$, satisfying the following conditions:
\begin{gather}
 [x, y, z] = -(-1)^{\bar x \bar y}[y, x, z], \quad [x, y, z] = -(-1)^{\bar y \bar z}[x, z, y], \label{3skewsuper:skewsym}\\
\begin{split} [x, y, [z, u, v]] = [[x, y, z], u, v] &+ (-1)^{\bar z(\bar x+\bar y)}[z, [x, y, u], v]  \\
 & + (-1)^{(\bar z+\bar u)(\bar x+\bar y)}[z, u, [x, y, v]],
\end{split} \label{3skewsuper}
\end{gather}
where $x, y, z, u, v \in\mathfrak g$ are homogeneous elements.
\end{defn}
\begin{prop} \label{pro:someequalities}
Let $\mathfrak g$ be a $\mathbb{Z}_2$-graded vector space  together with a super skew-symmetric even linear map $[\cdot,\cdot,\cdot]:
\otimes^3 \mathfrak g\rightarrow \mathfrak g$. Then $(\mathfrak g,[\cdot,\cdot,\cdot])$ is a $3$-Lie superalgebra if and only if the following identities hold:
\begin{align}
\begin{split} [[x_1,x_2,x_3],x_4, x_5]-(-1)^{\bar x_3\bar x_4}[[x_1,x_2,x_4],x_3,x_5]
+(-1)^{\bar x_2(\bar x_3+\bar x_4)}[[x_1,x_3,x_4],x_2,x_5] \\
-(-1)^{\bar x_1(\bar x_2+\bar x_3+\bar x_4)}[[x_2,x_3,x_4],x_1,x_5]=0,
\end{split}
\label{eq1:pro:someequalities}\\
\begin{split}
(-1)^{(\bar x_4+\bar x_5)(\bar x_1+\bar x_2+\bar x_3)}[[x_4,x_5,x_1],x_2,x_3]+(-1)^{(\bar x_4+\bar x_5)(\bar x_2+\bar x_3)}[x_1,[x_4,x_5,x_2],x_3]\\
    +(-1)^{(\bar x_1+\bar x_2)(\bar x_3+\bar x_4)}[[x_3,x_5,x_1],x_2,x_4]-(-1)^{(\bar x_3+\bar x_5)\bar x_2+\bar x_4\bar x_5}[x_1,[x_5,x_3,x_2],x_4]\\+(-1)^{\bar x_3(\bar x_4+\bar x_5)}[x_1,x_2,[x_4,x_5,x_3]]=0,
\end{split}
\label{eq2:pro:someequalities}
\end{align}
for all $x_i\in \mathcal H(\mathfrak g), 1\leq i\leq 5$. 
\end{prop}
\begin{proof}
If $(\mathfrak g,[\cdot,\cdot,\cdot])$ is a $3$-Lie superalgebra, then  applying \eqref{3skewsuper} to the last term in
\eqref{3skewsuper} for $x_i\in \mathcal H(\mathfrak g), 1\leq i\leq 5$ and using \eqref{3skewsuper:skewsym},
\begin{eqnarray*}
  [x_1,x_2,[x_3,x_4,x_5]]
  &=&[[x_1,x_2,x_3],x_4,x_5]+(-1)^{\bar x_3(\bar x_1+\bar x_2)}[x_3,[x_1,x_2,x_4],x_5]\\
 && +(-1)^{(\bar x_3+\bar x_4)(\bar x_1+\bar x_2)}[x_3,x_4,[x_1,x_2,x_5]]\\
  &=&[[x_1,x_2,x_3],x_4,x_5]+(-1)^{\bar x_3(\bar x_1+\bar x_2)}[x_3,[x_1,x_2,x_4],x_5]\\
  &&+(-1)^{(\bar x_3+\bar x_4)(\bar x_1+\bar x_2)}[[x_3,x_4,x_1],x_2,x_5] \\
  &&+(-1)^{\bar x_2(\bar x_3+\bar x_4)}[x_1,[x_3,x_4,x_2],x_5]\\
  &&+[x_1,x_2,[x_3,x_4,x_5]],
\end{eqnarray*}
yields \eqref{eq1:pro:someequalities}.\\
Similarly, applying \eqref{3skewsuper} to the first and second terms on the right hand side of \eqref{3skewsuper} for $x_i\in \mathcal H(\mathfrak g), 1\leq i\leq 5$ and using \eqref{3skewsuper:skewsym},
{\small\begin{eqnarray*}
 [x_1,x_2,[x_3,x_4,x_5]]
&=& (-1)^{(\bar x_4+\bar x_5)(\bar x_1+\bar x_2+\bar x_3)}[x_4,x_5,[x_1,x_2,x_3] \\
&& -(-1)^{(\bar x_3+\bar x_5)(\bar x_1+\bar x_2)+\bar x_4\bar x_5}[x_3,x_5,[x_1,x_2,x_4]]\\
 &&+(-1)^{(\bar x_3+\bar x_4)(\bar x_1+\bar x_2)}[x_3,x_4,[x_1,x_2,x_5]]\\
 &=&
 (-1)^{(\bar x_4+\bar x_5)(\bar x_1+\bar x_2+\bar x_3)}[[x_4,x_5,x_1],x_2,x_3] \\
 && +(-1)^{(\bar x_4+\bar x_5)(\bar x_2+\bar x_3)}[x_1,[x_4,x_5,x_2],x_3] \\
 &&+(-1)^{\bar x_3(\bar x_4+\bar x_5)}[x_1,x_2,[x_4,x_5,x_3]] \\
 && -(-1)^{(\bar x_3+\bar x_5)(\bar x_1+\bar x_2)+\bar x_4\bar x_5}[[x_3,x_5,x_1],x_2,x_4]\\
 &&-(-1)^{(\bar x_3+\bar x_5)\bar x_2+\bar x_4\bar x_5}[x_1,[x_3,x_5,x_2],x_4] \\
 && -(-1)^{\bar x_4\bar x_5}[x_1,x_2,[x_3,x_5,x_4]]\\
 &&+(-1)^{(\bar x_1+\bar x_2)(\bar x_3+\bar x_4)}[x_3,x_4,[x_1,x_2,x_5]],
\end{eqnarray*}}
yields \eqref{eq2:pro:someequalities}.

Conversely, suppose that \eqref{eq1:pro:someequalities} and \eqref{eq2:pro:someequalities} hold. First \eqref{eq1:pro:someequalities} gives
\begin{eqnarray*}
&& [[x_1,x_2,x_3],x_4, x_5]=(-1)^{\bar x_3\bar x_4}[[x_1,x_2,x_4],x_3,x_5]\\
&& \quad -(-1)^{\bar x_2(\bar x_3+\bar x_4)}[[x_1,x_3,x_4],x_2,x_5] + (-1)^{\bar x_1(\bar x_2+\bar x_3+\bar x_4)}[[x_2,x_3,x_4],x_1,x_5] \\
&& =(-1)^{(\bar x_3+\bar x_4+\bar x_5)(\bar x_1+\bar x_2)}[[x_3,x_4,x_5],x_1,x_2] \\
&& \quad +(-1)^{\bar x_1(\bar x_2+\bar x_4+\bar x_5)+\bar x_3(\bar x_4+\bar x_5)}[[x_2,x_4,x_5], x_1,x_3] + (-1)^{(\bar x_2+\bar x_3)(\bar x_4+\bar x_4)}[[x_1,x_4,x_5],x_2,x_3]
\end{eqnarray*}
by \eqref{eq2:pro:someequalities}.
Thus, $\mathfrak g$ is a $3$-Lie superalgebra.
\end{proof}

In the following  we recall that given a Lie superalgebra analogues of supertrace one can construct a $3$-Lie superalgebra \cite{Abramov2017:Super3LiealgebrasinducedsuperLiealg}. Let  $(\mathfrak g, [\cdot,\cdot])$ be a Lie superalgebra and $\tau:\mathfrak g\rightarrow \mathbb{K}$ an even linear form. We say that $\tau$ is a supertrace of $\mathfrak g$ if $\tau([\cdot,\cdot])=0$. For any $x_1, x_2,x_3 \in\mathcal{H}(\mathfrak g)$, we define
the 3-ary bracket by
\begin{equation}\label{crochet_n}
[x_1,x_2,x_3]_\tau=\tau(x_1)[x_2, x_3] - (-1)^{\bar x_1\bar x_2}\tau(x_2)[x_1, x_3]+(-1)^{\bar x_3 (\bar x_1+\bar x_2)}\tau(x_3)[x_1, x_2].
\end{equation}

\begin{thm}\label{ThmInduce1}
For any Lie superalgebra  $(\mathfrak g, [\cdot,\cdot])$ and supertrace $\tau$, the pair
$(\mathfrak g,[\cdot,\cdot,\cdot]_\tau)$ is a $3$-Lie superalgebra.
\end{thm}

\begin{defn}
Let $\mathfrak g$ be a $3$-Lie superalgebra,  $V$ be a graded vector space and
$\rho: \mathfrak g\wedge \mathfrak g\rightarrow \mathfrak{gl}(V)$ be an even linear mapping. If $\rho$ satisfies
\begin{eqnarray}\label{eq:mod1}
& \begin{aligned} \rho(x_1, x_2)\rho(x_3, x_4)-(-1)^{(\bar x_1+\bar x_2)(\bar x_3+\bar x_4)}\rho(x_3, x_4)\rho(x_1, x_2)\\
=\rho([x_1, x_2, x_3], x_4)-(-1)^{\bar x_3\bar x_4}\rho([x_1, x_2, x_4], x_3),
\end{aligned} \\
\label{eq:mod2}
& \begin{aligned}
\rho([x_1, x_2, x_3], x_4)=\rho(x_1, x_2)\rho(x_3, x_4)+(-1)^{\bar x_1(\bar x_2+\bar x_3)}\rho(x_2, x_3)\rho(x_1, x_4) \\
+(-1)^{\bar x_3(\bar x_1+\bar x_2)}\rho(x_3, x_1)\rho(x_2, x_4),
\end{aligned}
\end{eqnarray}
for all $x_i\in \mathcal H(\mathfrak g), 1\leq i\leq 4,$ then $(V, \rho)$ is called  a representation  of $\mathfrak g$, or $(V, \rho)$
is an $\mathfrak g$-module.
\end{defn}

Define \begin{equation}\label{eq:ad}
\mbox{ad}: \mathfrak g\wedge \mathfrak g\rightarrow \mathfrak{gl}(\mathfrak g),  ~~~~\mbox{ad}_{x, y}(z)=[x, y, z].~
\end{equation}
Using  \eqref{3skewsuper}, we can see that $(\mathfrak g, \mbox{ad})$ is a representation of the $3$-Lie superalgebra $\mathfrak g$  and it is called  the adjoint representation of $\mathfrak g$.

\begin{prop}\label{induced-rep}
Let $(V,\mu)$ be a representation of a Lie superalgebra $(\mathfrak g,[\cdot,\cdot])$ and $\tau$ be a supertrace of $\mathfrak g$. Then $(V, \rho_\tau)$ is a representation of the $3$-Lie superalgebra $(\mathfrak g, [\cdot,\cdot,\cdot]_\tau)$ where $\rho_\tau:\mathfrak g\wedge \mathfrak g\rightarrow \mathfrak{gl}(V)$ is defined by
\begin{equation}\label{rep-tau}
\rho_\tau(x,y)=\tau(x)\mu(y)-(-1)^{\bar x\bar y}\tau(y)\mu(x), \quad \quad \forall x,y\in \mathcal H(\mathfrak g).
\end{equation}
\end{prop}
\begin{proof}
For all $x_1,x_2,x_3,x_4\in  \mathcal H(\mathfrak g)$, the left hand of \eqref{eq:mod1} becomes
\begin{eqnarray*}
\begin{aligned}
& \rho_\tau(x_1, x_2)\rho_\tau(x_3, x_4)-(-1)^{(\bar x_1+\bar x_2)(\bar x_3+\bar x_4)}\rho_\tau(x_3, x_4)\rho_\tau(x_1, x_2)\\
& =\Big(\tau(x_1)\mu(x_2)-(-1)^{\bar x_1\bar x_2}\tau(x_2)\mu(x_1)\Big)\Big(\tau(x_3)\mu(x_4)-(-1)^{\bar x_3\bar x_4}\tau(x_4)\mu(x_3)\Big)
\\
& \quad -(-1)^{(\bar x_1+\bar x_2)(\bar x_3+\bar x_4)}\Big(\tau(x_3)\mu(x_4)-(-1)^{\bar x_3\bar x_4}\tau(x_4)\mu(x_3)\Big)\Big(\tau(x_1)\mu(x_2)-(-1)^{\bar x_1\bar x_2}\tau(x_2)\mu(x_1)\Big)\\
&=\tau(x_1)\mu(x_2)\tau(x_3)\mu(x_4)-(-1)^{\bar x_3\bar x_4}\tau(x_1)\mu(x_2)\tau(x_4)\mu(x_3)\\
&\quad -(-1)^{\bar x_1\bar x_2}\tau(x_2)\mu(x_1)\tau(x_3)\mu(x_4)-(-1)^{\bar x_1\bar x_2+\bar x_3\bar x_4}\tau(x_2)\mu(x_1)\tau(x_4)\mu(x_3)\\
&\quad -(-1)^{(\bar x_1+\bar x_2)(\bar x_3+\bar x_4)}\tau(x_3)\mu(x_4)\tau(x_1)\mu(x_2)+(-1)^{(\bar x_1+\bar x_2)(\bar x_3+\bar x_4)+\bar x_1\bar x_2}  \tau(x_3)\mu(x_4)\tau(x_2)\mu(x_1)\\
& \quad +(-1)^{(\bar x_1+\bar x_2)(\bar x_3+\bar x_4)+\bar x_3\bar x_4}  \tau(x_4)\mu(x_3)\tau(x_1)\mu(x_2)\\
& \quad -(-1)^{(\bar x_1+\bar x_2)(\bar x_3+\bar x_4)+\bar x_1\bar x_2+x_3\bar x_4}  \tau(x_4)\mu(x_3)\tau(x_2)\mu(x_1).
\end{aligned}
\end{eqnarray*}
Using \eqref{crochet_n} and \eqref{rep-tau}, the right hand of \eqref{eq:mod1} can be written as follows:
\begin{multline*}
\rho_\tau([x_1, x_2, x_3], x_4)-(-1)^{\bar x_3\bar x_4}\rho_\tau([x_1, x_2, x_4], x_3)\\
=-(-1)^{\bar x_4(\bar x_2+\bar x_3)}\tau(x_1)\tau(x_4)\mu([x_2,x_3])+(-1)^{\bar x_4(\bar x_2+\bar x_3)+\bar x_1\bar x_2}\tau(x_2)\tau(x_4)\mu([x_1,x_3])\\
-(-1)^{(\bar x_1+\bar x_2)(\bar x_3+\bar x_4)}\tau(x_3)\tau(x_4)\mu([x_1,x_2])+(-1)^{\bar x_2\bar x_3}\tau(x_1)\tau(x_3)\mu([x_2,x_4])\\
+(-1)^{\bar x_1\bar x_2+\bar x_3(\bar x_1+\bar x_4)}\tau(x_2)\tau(x_3)\mu([x_1,x_4])-(-1)^{(\bar x_1+\bar x_2)(\bar x_3+\bar x_4)}\tau(x_4)\tau(x_3)\mu([x_1,x_2]).
\end{multline*}
By direct identification, we proof that
\begin{multline*}
\rho_\tau(x_1, x_2)\rho_\tau(x_3, x_4)-(-1)^{(\bar x_1+\bar x_2)(\bar x_3+\bar x_4)}\rho_\tau(x_3, x_4)\rho_\tau(x_1, x_2)\\
=\rho_\tau([x_1, x_2, x_3], x_4)-(-1)^{\bar x_3\bar x_4}\rho_\tau([x_1, x_2, x_4], x_3).
\end{multline*}
Similarly, we can proof \eqref{eq:mod2}.
\end{proof}

We generalize the Lie-Rinehart superalgebra to the ternary case.
\begin{defn}\label{def-3Lie rinh super}
A $3$-Lie-Rinehart superalgebra over $A$ is a tuple $(L, A, [\cdot,\cdot,\cdot],\rho)$,
where $A$ is an associative supercommutative superalgebra, $L$ is an $A-$module, $[\cdot,\cdot, \cdot]: L\times L\times L \to L$ is
an even super skew-symmetric trilinear map, and the $\mathbb{K}$-map $\rho: L\times L \to Der(A)$
such that the following conditions hold:
\begin{enumerate}[label=\upshape{(\roman*)},left=7pt]
\item $(L, [\cdot,\cdot,\cdot])$ is a $3$-Lie superalgebra.
\item $\rho$ is a representation of $(L, [\cdot,\cdot,\cdot])$ on $A$.
\item  For all $x, y\in \mathcal H(L),\ a\in \mathcal H(A),$
\begin{equation}\label{eq:Rinhart2}
\rho(ax, y)=(-1)^{\bar a\bar x}\rho(x,ay)=a\rho(x, y).
\end{equation}
\item The compatibility condition:
\begin{equation}\label{eq:Rinhart1}
    [x, y, az]=(-1)^{\bar a(\bar x+\bar y)}a[x, y, z]+\rho(x, y)a z,~~\forall x, y, z\in \mathcal H(L),\  a\in \mathcal H(A).
\end{equation}
\end{enumerate}
\end{defn}
\begin{re}
If $\rho=0$, then  $(L, A,[\cdot,\cdot,\cdot])$ is called  a $3$-Lie $A$-superalgebra.
\end{re}
\begin{re}
If the condition \eqref{eq:Rinhart2} is not satisfied, then we call $(L, A, [\cdot,\cdot,\cdot],\rho)$ a weak $3$-Lie-Rinehart superalgebra over $A$.
\end{re}

\begin{defn}
Let $(L, A,[\cdot,\cdot,\cdot]_L, \rho)$ and $(L', A',[\cdot,\cdot,\cdot]_{L'}, \rho')$ be two $3$-Lie-Rinehart superalgebras,  then a $3$-Lie-Rinehart superalgebra homomorphism is defined as a pair of maps $(g, f)$, where $g: A \rightarrow A'$ and $f:L\rightarrow L'$ are two $\mathbb K$-algebra homomorphisms such that
\begin{enumerate}[label=\upshape{\arabic*)},left=7pt]
\item $f(ax) = g(a)f(x)$ for all $x \in L, a \in A,$
\item $g(\rho(x,y)(a)) = \rho'(f(x),f(y))(g(a))$ for all $x\in L, a \in A.$
\end{enumerate}

\end{defn}
\section{Some constructions of $3$-Lie-Rinehart superalgebras}
\label{sec:rbl3rbl}
In this section, we give some construction results. We begin by constructing $3$-Lie-Rinehart superalgebras starting with a Lie-Rinehart superalgebras. Then, we establish the converse sense. At the end of this section, we construct some new $3$-Lie-Rinehart superalgebras from a given $3$-Lie-Rinehart
superalgebra.

The following Theorem generalizes results in \cite{Abramov2019:3LiesuperalgsLiesuperals} to Lie-Rinehart case.
\begin{thm}\label{ThmInduce2}
Let $(L,A,[\cdot,\cdot],\mu)$ be a Lie-Rinehart superalgebra and $\tau$ is a supertrace. If the condition
\begin{equation}
  \tau(ax)y=\tau(x)ay  \label{ConditionTau1}
   \end{equation}
is satisfied for any $x,y\in\mathcal H(L)$, $a\in \mathcal H( A)$, then $(L, A,[\cdot,\cdot,\cdot]_\tau,\rho_\tau)$ is a $3$-Lie-Rinehart superalgebra, where $[\cdot,\cdot,\cdot]_\tau$ and  $\rho_\tau$ are defined in \eqref{crochet_n} and \eqref{induced-rep} respectively.
We say that $(L, A,[\cdot,\cdot,\cdot]_\tau,\rho_\tau)$  is induced by $(L,A,[\cdot,\cdot],\mu)$ and is denoted by $L_\tau$.

\end{thm}
\begin{proof}
Proposition \ref{induced-rep} gives that $(A,\rho_\tau)$ is a representation of the $3$-Lie superalgebra $(L, [\cdot,\cdot,\cdot]_\tau)$ and $\rho_\tau(x,y)\in Der(A)$. For all $x,y,z\in\mathcal H(L)$ and $a\in\mathcal H(A)$,
\begin{eqnarray*}
[x,y,az]_\tau&=&\tau(x)[y,az]-(-1)^{\bar x\bar y}\tau(y)[x,az]+(-1)^{(\bar a+\bar z)(\bar x+\bar y)}\tau(az)[x,y]\\
&=&(-1)^{\bar a(\bar x+\bar y)} a\tau(x)[y,z]+\tau(x)\mu(y)az
-(-1)^{\bar a(\bar x+\bar y)+\bar x\bar y} a\tau(y)[x,z] \\
&& -(-1)^{\bar x\bar y}\tau(y)\mu(x)az +(-1)^{(\bar a+\bar z)(\bar x+\bar y)}a\tau(z)[x,y]\\
&=&(-1)^{\bar a(\bar x+\bar y)}a[x,y,z]_\tau+\rho_\tau(x,y)az.
\end{eqnarray*}
So, we obtain \eqref{eq:Rinhart1}.
Since \eqref{ConditionTau1} is satisfied and
\begin{eqnarray*}
\rho_\tau(ax,y)&=&\tau(ax)\mu(y)-(-1)^{(\bar a+\bar x)\bar y}\tau(y)\mu(ax)\ ,\\
(-1)^{\bar a\bar x}\rho_\tau(x,ay)&=&(-1)^{\bar a\bar x}\tau(x)\mu(ay)-(-1)^{(\bar a+\bar y)\bar x+\bar a\bar x}\tau(ay)\mu(x)
\ ,\\
a\rho_\tau(x,y)&=&a\tau(x)\mu(y)-(-1)^{\bar x\bar y}a\tau(y)\mu(x) \ ,
\end{eqnarray*}
the condition \eqref{eq:Rinhart2} holds. \end{proof}

Conversely,  we can construct a Lie-Rinehart superalgebra structure from a given $3$-Lie-Rinehart superalgebra.

\begin{prop}
Let $(L,A,[\cdot,\cdot,\cdot],\rho)$ be a $3$-Lie-Rinehart superalgebra. Let ${x_0}\in L_{\overline 0}$. Define the bracket
$[\cdot,\cdot]_{x_0} = [{x_0},\cdot,\cdot]$ and $\rho_{x_0}(x)(v)=\rho(x_0,x)v$.  Then $(L,A,[\cdot,\cdot]_{x_0},\rho_{x_0})$ is a Lie-Rinehart superalgebra.

\end{prop}
\begin{proof}
It is easy to check that the bracket is super skew-symmetric. For any $x,y,z\in \mathcal H(L)$, we have
\begin{eqnarray*}
[x,[y,z]_{x_0}]_{x_0}&=&[{x_0},x,[{x_0},y,z]]\\
&=&[[{x_0},x,{x_0}],y,z] + [{x_0},[{x_0},x,y],z] +(-1)^{\bar x\bar y} [{x_0},y,[{x_0},x,z]]\\
&=& [{x_0},[{x_0},x,y],z] + (-1)^{\bar x\bar y}[{x_0},y,[{x_0},x,z]]\\
&=& [[x,y]_{x_0} ,z]_{x_0} + (-1)^{\bar x\bar y}[y,[x,z]_{x_0}]_{x_0}.
\end{eqnarray*}
It is obvious to see  that $\rho_{x_0}$ is a representation of $L$ on $A$ and  $\rho_{x_0}(x)\in Der(A)$. It remains to prove  \eqref{compatibility-rinhe-super}. Using \eqref{eq:Rinhart1}, we obtain
\begin{eqnarray*}
[x, ay]_{x_0} = [x_0,x,ay] = (-1)^{\bar a\bar x}a[x_0,x,y]+\rho(x_0,x)ay\\
=(-1)^{\bar a\bar x}a[x,y]_{x_0}+\rho_{x_0}(x)ay.
\end{eqnarray*}
which completes the proof.
\end{proof}
\begin{defn}\label{defn:subandideal} Let $(L,A,[\cdot,\cdot,\cdot],\rho)$ be a  $3$-Lie-Rinehart superalgebra.
\begin{enumerate}[label=\upshape{\arabic*)},left=7pt]
  \item  If $S$ is a subalgebra of the $3$-Lie superalgebra $(L, [\cdot,\cdot,\cdot])$ satisfying $AS\subset S$, then   $(S, A, [\cdot,\cdot,\cdot],\rho|_{S\wedge S})$ is a  $3$-Lie-Rinehart superalgebra, which is called  a subalgebra of the $3$-Lie-Rinehart superalgebra $(L,A,[\cdot,\cdot,\cdot],\rho)$.
\item If $I$ is an ideal of the $3$-Lie superalgebra $(L,[\cdot,\cdot,\cdot])$ and satisfies $AI\subset I$ and $\rho(I,L)(A)L\subset I$, then   $(I, A,[\cdot,\cdot,\cdot],\rho|_{I\wedge I})$ is a $3$-Lie-Rinehart superalgebra, which is called  an ideal of the $3$-Lie-Rinehart superalgebra $(L,A,[\cdot,\cdot,\cdot],\rho)$.
 \item If a $3$-Lie-Rinehart superalgebra $(L,A,[\cdot,\cdot,\cdot],\rho)$  cannot be decomposed into the direct sum of two nonzero ideals, then $L$ is called an indecomposable $3$-Lie-Rinehart superalgebra.
     \end{enumerate}
\end{defn}

\begin{prop}\label{prop:ker} If  $(L,A,[\cdot,\cdot,\cdot],\rho)$ is  a  $3$-Lie-Rinehart superalgebra. Then
 $$Ker \rho=\{x \in \mathcal H(L) \mid \rho(x,L)= 0\}$$
  is an ideal, which is called the kernel of the representation $\rho$.
\end{prop}

\begin{proof}
By \eqref{eq:mod1} and \eqref{eq:mod2},  for all homogeneous elements  $x\in Ker\rho$, $y,z,w\in \mathcal H(L)$,
$$
    \rho([x,y,z],w)=\rho(x,y)\rho(z,w)+(-1)^{\bar x(\bar y+\bar z)}\rho(y,z)\rho(x,w)+(-1)^{\bar z(\bar x+\bar y)}\rho(z,x)\rho(y,w)=0.
$$
Therefore, $[x, L, L]\subseteq Ker\rho$. By \eqref{eq:Rinhart2}, for all $a\in \mathcal H(A)$,
$$\rho(ax,y)=(-1)^{\bar x\bar a}\rho(x,ay)=a\rho(x,y)=0, \quad\rho(\rho(x,y)az,w)=\rho(x,y)(a)\rho(z,w)=0.$$
We get  $ax\in Ker\rho$, $\rho(x,y)az\in Ker\rho$, that is, $A Ker\rho\subset Ker\rho$, $\rho(Ker\rho,L)(A)L\subset Ker\rho$.

Therefore, $Ker\rho$ is an ideal of the  $3$-Lie-Rinehart superalgebra $(L,A,[\cdot,\cdot,\cdot],\rho)$.
\end{proof}

\begin{thm}
Let  $(L,A,[\cdot,\cdot,\cdot],\rho)$ be a $3$-Lie-Rinehart superalgebra. Then the following identities hold,  for all $a,b,c\in \mathcal H(A), x_i\in \mathcal H(L), 1\leq i\leq 5,$
\begin{equation}\label{eq:ho1}
\begin{split}
&\rho(x_2, x_3)a[x_1, x_4, x_5] +(-1)^{(\bar x_1+\bar x_4)(\bar x_2+\bar x_3)+\bar a(\bar x_1+\bar x_4+\bar x_2+\bar x_3)} \rho(x_1, x_4)a[x_2, x_3, x_5] \\
&+ (-1)^{\bar x_2(\bar x_3+\bar x_1)+\bar a(\bar x_1+\bar x_2)} \rho(x_3, x_1)a[x_2, x_4, x_5]\\
&+(-1)^{\bar x_4(\bar x_3+\bar x_1)+\bar a(\bar x_3+\bar x_4)} \rho(x_2, x_4)a[x_3, x_1, x_5] \\
&+(-1)^{\bar x_1(\bar x_2+\bar x_3)+\bar a(\bar x_1+\bar x_3)}  \rho(x_1, x_2)a[x_3, x_4, x_5] \\
&+(-1)^{\bar x_2(\bar x_1+\bar x_3+\bar x_4)+\bar x_4\bar x_1+\bar a(\bar x_2+\bar x_4))}  \rho(x_3, x_4)a[x_1, x_2, x_5]=0.
 \end{split}
  \end{equation}
\begin{equation}\label{eq:ho2}
\begin{split}
&\rho(x_2, x_3)(a_4a_5)[x_1, x_4, x_5] +(-1)^{\bar x_2(\bar a_1+\bar a_4+\bar a_5+\bar x_3)} \rho(x_3, x_1)(a_4a_5)[x_2, x_4, x_5] \\
&+(-1)^{\bar x_1(\bar a_4+\bar a_5+\bar x_2+\bar x_3)} \rho(x_1, x_2)(a_4a_5)[x_3, x_4, x_5]\\
&+ (-1)^{(\bar a_4+\bar x_5)(\bar x_2+\bar x_3)+(\bar x_1+\bar x_4)(\bar x_2+\bar x_3+\bar a_5)}a_4\rho(x_1, x_4)a_5[x_5, x_2, x_3] \\
&+(-1)^{\bar a_4(\bar x_2+\bar x_3)+(\bar x_4+\bar x_5)(\bar x_1+\bar x_3)+\bar x_4\bar a_5} a_4\rho(x_2, x_4)a_5[x_5, x_3, x_1]\\
 &+ (-1)^{\bar a_4(\bar x_2+\bar x_3)+(\bar x_4+\bar x_5)(\bar x_1+\bar x_2)+\bar x_4\bar a_5}a_4\rho(x_3, x_4)a_5[x_5, x_1, x_2] \\
&-(-1)^{(\bar x_1+\bar a_5+\bar x_5)(\bar a_4+\bar x_2+\bar x_3)+\bar x_4(\bar x_2+\bar x_3)} a_5\rho(x_1, x_5)a_4[x_4, x_2, x_3] \\
 &-(-1)^{(\bar a_5+\bar x_5)(\bar x_2+\bar x_3)+(\bar a_4+\bar x_5)(\bar x_3+\bar x_5)}a_5\rho(x_2, x_5)a_4[x_4, x_3, x_1] \\
&-(-1)^{\bar a_5(\bar x_2+\bar x_3+\bar a_4)+(\bar x_2+\bar x_5)(\bar x_1+\bar x_4+\bar a_4)+\bar x_2(\bar x_3+\bar x_5)+\bar x_1\bar x_4} a_5 \rho(x_3, x_5)a_4[x_4, x_1, x_2]) = 0.
\end{split}
\end{equation}
\begin{equation}\label{eq:ho4}
\begin{split}
&\rho(x_1, x_2)\rho(x_3, x_4) + (-1)^{(\bar x_1+\bar x_2)(\bar x_3+\bar x_4)}\rho(x_3, x_4)\rho(x_1, x_2) \\
&+ (-1)^{x_1(\bar x_2+\bar x_3)}\rho(x_2, x_3)\rho(x_1, x_4)+ (-1)^{\bar x_4(\bar x_2+\bar x_3)}\rho(x_1, x_4)\rho(x_2, x_3)\\
& + (-1)^{\bar x_3(\bar x_1+\bar x_2)}\rho(x_3, x_1)\rho(x_2, x_4) + (-1)^{\bar x_1(\bar x_2+\bar x_3+\bar x_4)+\bar x_3\bar x_4}\rho(x_2, x_4)\rho(x_3, x_1)=0.
\end{split}
 \end{equation}
\begin{equation}\label{eq:ho5}
\begin{split}
&(\rho(x_1, x_2)b)\rho(x_3, x_4) +(-1)^{\bar x_3(\bar x_1+\bar x_2+\bar b)+\bar b\bar x_2} (\rho(x_3, x_1)b)\rho(x_2, x_4)\\
&+ (-1)^{\bar x_1(\bar x_2+\bar x_3+\bar b)+\bar x_1\bar b}(\rho(x_2, x_3)b)\rho(x_1, x_4)=0.
\end{split}
  \end{equation}
\end{thm}

\begin{proof}
Using \eqref{eq:Rinhart1} we have
\begin{eqnarray*}
&&\rho(x_2, x_3)a[x_1, x_4, x_5] +
(-1)^{(\bar x_1+\bar x_4)(\bar x_2+\bar x_3)+\bar a(\bar x_1+\bar x_4+\bar x_2+\bar x_3)} \rho(x_1, x_4)a[x_2, x_3, x_5] \\
&&+ (-1)^{\bar x_2(\bar x_3+\bar x_1)+\bar a(\bar x_1+\bar x_2)} \rho(x_3, x_1)a[x_2, x_4, x_5]\\
&&+(-1)^{\bar x_4(\bar x_3+\bar x_1)+\bar a(\bar x_2+\bar x_4)} \rho(x_2, x_4)a[x_3, x_1, x_5] \\
&&+(-1)^{\bar x_1(\bar x_2+\bar x_3)+\bar a(\bar x_1+\bar x_3)}  \rho(x_1, x_2)a[x_3, x_4, x_5] \\
&&+(-1)^{\bar x_2(\bar x_1+\bar x_3+\bar x_4)+\bar x_4\bar x_1+\bar a(\bar x_2+\bar x_4))}  \rho(x_3, x_4)a[x_1, x_2, x_5]\\
&&=[x_2,x_3,a[x_1,x_4,x_5]]-\underbrace{(-1)^{\bar a(\bar x_2+\bar x_3)}a[x_2,x_3,[x_1,x_4,x_5]]}_{I_1\label{1}}\\
&&+(-1)^{(\bar x_1+\bar x_4)(\bar x_2+\bar x_3)+\bar a(\bar x_1+\bar x_4+\bar x_2+\bar x_3)}\Big([x_1,x_4,a[x_2,x_3,x_5]]\underbrace{-(-1)^{\bar a(\bar x_1+\bar x_4)}a[x_1,x_4,[x_2,x_3,x_5]]}_{I_2\label{2}}\Big)\\
&&+(-1)^{\bar x_2(\bar x_3+\bar x_1)+\bar a(\bar x_1+\bar x_2)} \Big([x_3,x_1,a[x_2,x_4,x_5]]\underbrace{-(-1)^{\bar a(\bar x_3+\bar x_1)}a[x_4,x_1,[x_2,x_4,x_5]]}_{I_3\label{3}}\Big)\\
&&+(-1)^{\bar x_4(\bar x_3+\bar x_1)+\bar a(\bar x_3+\bar x_4)} \Big([x_2,x_4,a[x_3,x_1,x_5]]\underbrace{-(-1)^{\bar a(\bar x_2+\bar x_4)}a[x_2,x_4,[x_3,x_1,x_5]]}_{I_4\label{4}}\Big)\\
&&+(-1)^{\bar x_1(\bar x_2+\bar x_3)+\bar a(\bar x_1+\bar x_3)} \Big([x_1,x_2,a[x_3,x_4,x_5]]\underbrace{-(-1)^{\bar a(\bar x_1+\bar x_2)}a[x_1,x_2,[x_3,x_4,x_5]]}_{I_5\label{5}}\Big)\\
&&+(-1)^{\bar x_2(\bar x_1+\bar x_3+\bar x_4)+\bar x_4\bar x_1+\bar a(\bar x_2+\bar x_4))}
\Big([x_3,x_4,a[x_1,x_2,x_5]]-\underbrace{(-1)^{\bar a(\bar x_3+\bar x_4)}a[x_3,x_4,[x_1,x_2,x_5]]}_{I_6\label{6}}\Big)
  \end{eqnarray*}
Thanks to Proposition \ref{pro:someequalities}, we have $\sum_{i=1}^6 I_i=0$. Using \eqref{eq:Rinhart1}, \eqref{eq:mod1} and \eqref{eq:mod2}, we obtain \eqref{eq:ho1}.
Similarly, \eqref{eq:ho2}- \eqref{eq:ho5} can be verified by a direct computation according to \eqref{eq:ho1}
and Definition \ref{def-3Lie rinh super}.
\end{proof}

Now we construct some new $3$-Lie-Rinehart superalgebra from a given $3$-Lie-Rinehart superalgebra
$(L, A, [\cdot,\cdot,\cdot],\rho).$
\begin{thm}
 Let $(L, A, [\cdot,\cdot,\cdot],\rho)$ be a $3$-Lie-Rinehart superalgebra and $B=A\otimes L=\{ax ~|~ a\in  \mathcal H(A), x\in  \mathcal H(L)\}$.
Then $(B, A, [\cdot,\cdot,\cdot]',\rho')$ is a $3$-Lie-Rinehart superalgebra, where the multiplication is defined by, for all $x_1, x_2, x_3\in \mathcal H(L)$,
$a_1, a_2, a_3\in \mathcal H(A),$
{\small{\begin{equation}
\begin{split}
[a_1x_1, a_2x_2, a_3x_3]' =& (-1)^{\bar a_2\bar x_1+\bar a_3(\bar x_1+\bar x_2)}a_1a_2a_3[x_1, x_2, x_3] +(-1)^{\bar a_2\bar x_1} a_1a_2\rho(x_1,x_2)(a_3)x_3 \\
&+(-1)^{\bar a_3\bar x_2+(\bar a_1+\bar x_1)(\bar a_2+\bar a_3+\bar x_2+\bar x_3)} a_2a_3\rho(x_2, x_3)(a_1)x_1\\
&+(-1)^{\bar a_3\bar x_1+(\bar a_2+x_2)(\bar a_3+\bar x_3)} a_1a_3\rho(x_1, x_3)(a_2)x_2,
\end{split}
\end{equation}}}
and $\rho':(A\otimes L) \wedge (A \otimes L) \rightarrow Der(A)$ is defined by $\rho'(a_1x_1, a_2x_2) =(-1)^{\bar a_2\bar x_1} a_1a_2\rho(x_1,x_2)$.

Note that $B=B_{\overline0}\oplus B_{\overline1}$ where $B_{\overline0}=A_{\overline0}\otimes L_{\overline0}\oplus A_{\overline1}\otimes L_{\overline1}$ and $B_{\overline1}=A_{\overline0}\otimes L_{\overline1}\oplus A_{\overline1}\otimes L_{\overline0}$.
\end{thm}
\begin{proof}
For the super skew-symmetry of the bracket we have
\begin{eqnarray*}
[a_1x_1, a_2x_2, a_3x_3]' &=& (-1)^{\bar a_2\bar x_1+\bar a_3(\bar x_1+\bar x_2)}a_1a_2a_3[x_1, x_2, x_3]+(-1)^{\bar a_2\bar x_1} a_1a_2\rho(x_1,x_2)(a_3)x_3 \\
&&+(-1)^{\bar a_3\bar x_2+(\bar a_1+\bar x_1)(\bar a_2+\bar a_3+\bar x_2+\bar x_3)} a_2a_3\rho(x_2, x_3)(a_1)x_1\\
&&+(-1)^{\bar a_3\bar x_1+(\bar a_2+x_2)(\bar a_3+\bar x_3)} a_1a_3\rho(x_1, x_3)(a_2)x_2\\
&=&-(-1)^{(\bar a_1+\bar x_1)(\bar a_2+\bar x_2)}\Big((-1)^{\bar a_1\bar x_2+\bar a_3(\bar x_1+\bar x_2)}a_2a_1a_3[x_2, x_1, x_3] \\
&&+(-1)^{\bar a_1\bar x_2} a_2a_1\rho(x_2,x_1)(a_3)x_3 \\
&&+(-1)^{\bar a_3\bar x_1+(\bar a_2+\bar x_2)(\bar a_1+\bar a_3+\bar x_2+\bar x_3)} a_1a_3\rho(x_1, x_3)(a_2)x_2\\
&&+(-1)^{\bar a_3\bar x_2+(\bar a_1+x_1)(\bar a_3+\bar x_3)} a_2a_3\rho(x_2, x_3)(a_1)x_1\Big)\\
&=&-(-1)^{(\bar a_1+\bar x_1)(\bar a_2+\bar x_2)}[a_2x_2, a_1x_1, a_3x_3]'.
\end{eqnarray*}
Similarly, we obtain  $[a_1x_1, a_2x_2, a_3x_3]' =-(-1)^{(\bar a_2+\bar x_2)(\bar a_3+\bar x_3)}[a_1x_1, a_3x_3, a_2x_2]'$.

Using  \eqref{eq:ho1} and \eqref{eq:ho2}, we can deduce that $(B,[\cdot,\cdot,\cdot]')$ is a $3$-Lie superalgebra and $B$ is an
$A$-module. Thanks to \eqref{eq:mod1} and \eqref{eq:mod2}, for all $x_1, x_2, x_3 \in \mathcal H(L)$ and $b, a_1, a_2, a_3\in\mathcal H(A),$ it is easy to show that  $(A, \rho')$ is a representation of  $B$. Indeed,
\begin{alignat*}{2}
&\rho'(a_1x_1, a_2x_2)\rho'(a_3x_3, a_4x_4)-(-1)^{(\bar a_1+\bar x_1+\bar a_2+\bar x_2)(\bar x_3+\bar a_3+\bar x_4+\bar x_4)}
\rho'(a_3x_3, a_4x_4)\rho'(a_1x_1, a_2x_2)\\
& \quad =\rho'([a_1x_1, a_2x_2, a_3x_3]', a_4x_4) +(-1)^{(\bar a_1+\bar x_1+\bar a_2+\bar x_2)(\bar a_3+\bar x_3)} \rho'(a_3x_3, [a_1x_1, a_2x_2, a_4x_4]'), \\
&\rho'(a_1x_1, a_2x_2)\rho'(a_3x_3, a_4x_4) +(-1)^{(\bar a_1+\bar x_1)(\bar a_2+\bar x_2+\bar x_3+\bar x_3)} \rho'(a_2x_2, a_3x_3)\rho'(a_1x_1, a_4x_4)\\
& \quad + (-1)^{(\bar a_1+\bar x_1+\bar a_2+\bar x_2)(\bar a_3+\bar x_3)}\rho'(a_3x_3, a_1x_1)\rho'(a_2x_2, a_4x_4)=\rho'([a_1x_1, a_2x_2, a_3x_3], a_4x_4).
\end{alignat*}

To prove the compatibility condition \eqref{eq:Rinhart1}, we compute as follows
\begin{eqnarray*}
&&[a_1x_1, a_2x_2, b(a_3x_3)]' =(-1)^{\bar a_2\bar x_1+(\bar a_3+\bar b)(\bar x_1+\bar x_3)} a_1a_2ba_3[x_1, x_2, x_3]\\
&& \quad +(-1)^{\bar a_2\bar x_1 } a_1a_2\rho(x_1,x_2)(ba_3)x_3 \\
&& \quad +(-1)^{(\bar b+\bar a_3)\bar x_2+(\bar a+\bar x_1)(\bar a_2+\bar x_2+\bar a_3+\bar x_3)} a_2ba_3\rho(x_2, x_3)a_1x_1\\
&& \quad +(-1)^{(\bar b+\bar a_3)\bar x_1+(\bar a_2+\bar x_2)(\bar a_3+\bar b+\bar x_3}a_1ba_3\rho(x_1,x_3)a_2x_2\\
&&=(-1)^{\bar b(\bar a_1+\bar x_1+\bar a_2+\bar x_2)}b\Big(a_1a_2a_3[x_1, x_2, x_3] + (-1)^{\bar a_2\bar x_1 }a_1a_2\rho(x_1,x_2)a_3x_3 \\
&& \quad +(-1)^{\bar a_3\bar x_2+(\bar a_1+\bar x_1)(\bar a_2+\bar x_2+\bar a_3+\bar x_3)} a_2a_3\rho(x_2, x_3)a_1x_1\\
&& \quad +(-1)^{\bar a_3\bar x_1+(\bar a_2+\bar x_2)(\bar a_3+\bar x_3)}a_1a_3\rho(x_1,x_3)a_2x_2\Big)+(-1)^{\bar a_2\bar x_1}a_1a_2\rho(x_1,x_2)b(a_3x_3)\\
&&=(-1)^{\bar b(\bar a_1+\bar x_1+\bar a_2+\bar x_2)}b[a_1x_1, a_2x_2, a_3x_3]' + \rho'(a_1x_1, a_2x_2)b(a_3x_3).
\end{eqnarray*}
It is obvious to show that
$$
\rho'(b(a_1x_1),a_2x_2)=b\rho'(a_1x_1,a_2x_2)=(-1)^{\bar b(\bar a+\bar x_1)}\rho'(a_1x_1,b(a_2x_2)).
$$
Therefore, $(B, A, [\cdot,\cdot,\cdot]',\rho')$ is a $3$-Lie-Rinehart superalgebra.
\end{proof}
\begin{thm}
Let $(L, A,[\cdot,\cdot,\cdot],\rho)$ be a $3$-Lie-Rinehart superalgebra and
$$E=L\oplus A=\{(x,a)\  \ |\ \ x\in L, a\in A\}.$$ Then
$(E, A, [\cdot,\cdot,\cdot]_1,\rho_1)$ is a $3$-Lie-Rinehart superalgebra, where for any $a, b, c\in\mathcal H(A)$,
$x, y, z\in\mathcal H(L)$ and  $k \in \mathbb K$,
\begin{equation}\label{eq16}
 k(x, a) = (kx, ka), (x, a) + (y, b) = (x + y, a + b), a(y, b) = (ay, ab),
 \end{equation}
 \begin{equation}\label{eq17}
 [(x, a), (y, b), (z, c)]_1 = ([x, y, z], \rho(x, y)c -(-1)^{\bar y\bar z}\rho(x, z)b+(-1)^{\bar x(\bar y+\bar z)} \rho(y, z)a ),
 \end{equation}
 \begin{equation}\label{eq18}
\rho_1 : E\wedge E \rightarrow Der(A), \rho_1((x, a),(y, b)) = \rho(x, y).
\end{equation}

Note that $E=E_{\overline0}\oplus E_{\overline1}$, where $E_{\overline0}= L_{\overline0}\oplus A_{\overline0}$ and $E_{\overline1}= L_{\overline1}\oplus A_{\overline1}$.
\end{thm}
\begin{proof}
Let  $x_i \in\mathcal H(L), a_i \in \mathcal H(A), i=1,\dots, 5$. Thanks to \eqref{eq16}, $E$ is an $A-$module, and the 3-ary linear multiplication defined by \eqref{eq17} is  super skew-symmetric. We have
\begin{align*}
& [(x_1, a_1), (x_2, a_2), [(x_3, a_3), (x_4, a_4), (x_5, a_5)]_1]_1 \\
& =[(x_1, a_1), (x_2, a_2),([x_3,x_4,x_5],\rho(x_3,x_4)a_5  -(-1)^{\bar x_4\bar x_5}\rho(x_3,x_5)a_4 \\
& \hspace{7cm} +(-1)^{\bar x_3(\bar x_4+\bar x_5)}\rho(x_4,x_5)a_3)]_1 \\
& =\big([x_1,x_2,[x_3,x_4,x_5]],\rho(x_1,x_2)\rho(x_3,x_4)a_5-(-1)^{\bar x_4\bar x_5}\rho(x_1,x_2)\rho(x_3,x_5)a_4\\
& +(-1)^{\bar x_3(\bar x_4+\bar x_5)}\rho(x_1,x_2)\rho(x_4,x_5)a_3-(-1)^{\bar x_2(\bar x_3+\bar x_4+\bar x_5)}\rho(x_1,[x_3,x_4,x_5])a_2\\
& -(-1)^{\bar x_1(\bar x_2+\bar x_3+\bar x_4+\bar x_5)}\rho(x_2,[x_3,x_4,x_5])a_1\big).
\end{align*}
Similarly we obtain
\begin{align*}
&[[(x_1, a_1), (x_2, a_2), (x_3, a_3)]_1, (x_4, a_4), (x_5, a_5)]_1\\
&+(-1)^{\bar x_3(\bar x_1+\bar x_2)}[(x_3, a_3),[(x_1, a_1), (x_2, a_2),  (x_4, a_4)]_1, (x_5, a_5)]_1\\
&+(-1)^{(\bar x_3+\bar x_4)(\bar x_1+\bar x_2)}[(x_3, a_3), (x_4, a_4),[(x_1, a_1), (x_2, a_2), (x_5, a_5)]_1]_1\\
&=[([x_1,x_2,x_3],\rho(x_1,x_2)a_3-(-1)^{\bar x_2\bar x_3}\rho(x_1,x_3)a_2 \\
&+(-1)^{\bar x_1(\bar x_2+\bar x_3)}\rho(x_2,x_3)a_2), (x_4, a_4), (x_5, a_5)]_1\\
&+(-1)^{\bar x_3(\bar x_1+\bar x_2)} \big([(x_3,a_3), ([x_1,x_2,x_4],\rho(x_1,x_2)a_4-(-1)^{\bar x_2\bar x_4}\rho(x_1,x_4)a_2\\
&+(-1)^{\bar x_1(\bar x_2+\bar x_4)}\rho(x_2,x_4)a_2), (x_5, a_5)]_1\big)\\
&+ (-1)^{(\bar x_3+\bar x_4)(\bar x_1+\bar x_2)}\big([(x_3,a_3), (x_4,a_5), ([x_1,x_2,x_5],\rho(x_1,x_2)a_5-(-1)^{\bar x_2\bar x_5}\rho(x_1,x_5)a_2\\
&+(-1)^{\bar x_1(\bar x_2+\bar x_5)}\rho(x_2,x_5)a_2)]_1\big)\\
&=\big([x_1,x_2,[x_3,x_4,x_5]],\rho(x_1,x_2)\rho(x_3,x_4)a_5-(-1)^{\bar x_4\bar x_5}\rho(x_1,x_2)\rho(x_3,x_5)a_4\\
&+(-1)^{\bar x_4(\bar x_4+\bar x_5)}\rho(x_1,x_2)\rho(x_4,x_5)a_3-(-1)^{\bar x_2(\bar x_3+\bar x_4+\bar x_5)}\rho(x_1,[x_3,x_4,x_5])a_2\\
&-(-1)^{\bar x_1(\bar x_2+\bar x_3+\bar x_4+\bar x_5)}\rho(x_2,[x_3,x_4,x_5])a_1\big)\\
&+(-1)^{\bar x_3(\bar x_1+\bar x_2)}\Big([x_3,[x_1,x_2,x_4],x_5],\rho(x_3,[x_1,x_2,x_4])a_5 \\
&-(-1)^{\bar x_5(\bar x_1+\bar x_2+\bar x_4)}(\rho(x_3,x_5)\rho(x_1,x_2)a_4-(-1)^{\bar x_2\bar x_4}\rho(x_3,x_5)\rho(x_1,x_4)a_2\\
&+(-1)^{\bar x_1(\bar x_2+\bar x_4)}\rho(x_3,x_5)\rho(x_2,x_4)a_1)+(-1)^{\bar x_3(\bar x_1+\bar x_2+\bar x_4+\bar x_5)}\rho([x_1,x_2,x_4],x_5)a_3\Big)\\
&+(-1)^{(\bar x_3+\bar x_4)(\bar x_1+\bar x_2)}\Big([x_3,x_4,[x_1,x_2,x_5]],
\rho(x_3,x_4)\rho(x_1,x_2)a_5-(-1)^{\bar x_2\bar x_5}\rho(x_3,x_4)\rho(x_1,x_5)a_2\\
&+(-1)^{\bar x_1(\bar x_2+\bar x_5)}\rho(x_3,x_4)\rho(x_2,x_5)a_1- (-1)^{\bar x_4(\bar x_1+\bar x_2+\bar x_5)}\rho(x_3,[x_1,x_2,x_5])a_4 \\
&+ (-1)^{\bar x_3(\bar x_1+\bar x_2+\bar x_4+\bar x_5)}\rho(x_4,[x_1,x_2,x_5])a_3\Big).
\end{align*}
Then \eqref{3skewsuper} holds thanks to \eqref{eq:mod1}. Therefore, $(E,[\cdot,\cdot,\cdot]_1)$ is a $3$-Lie superalgebra.

Now we prove that $\rho_1$ is a representation of $E$ over $A$. By \eqref{eq18}, we have
\begin{align*}
& \rho_1([(x_1, a_1), (x_2, a_2), (x_3, a_3)], (x_4, a_4)) -(-1)^{\bar x_3\bar x_4} \rho_1((x_3, a_3), [(x_1, a_1), (x_2, a_2), (x_4, a_4)])\\
& =\rho_1(([x_1, x_2, x_3], \rho(x_1, x_2)a_3 -(-1)^{\bar x_2\bar x_3} \rho(x_1, x_2)a_3 +(-1)^{\bar x_1(\bar x_2+\bar x_3)} \rho(x_2, x_3)a_1), (x_4, a_4))\\
&-(-1)^{\bar x_3\bar x_4} \Big( \rho_1((x_3, a_3),([x_1, x_2, x_4], \rho(x_1, x_2)a_4 -(-1)^{\bar x_2\bar x_4} \rho(x_1, x_2)a_4 \\
&+(-1)^{\bar x_1(\bar x_2+\bar x_4)} \rho(x_2, x_4)a_1))\Big)\\
&=\rho([x_1, x_2, x_3], x_4) -(-1)^{\bar x_3\bar x_4} \rho(x_3, [x_1, x_2, x_4]) \\
&=\rho(x_1,x_2)\rho(x_3,x_4)-(-1)^{(\bar x_1+\bar x_2)(\bar x_3+\bar x_4)}\rho(x_3,x_4)\rho(x_1,x_2)\\
&=\rho_1((x_1,a_1),(x_2,a_2))\rho_1((x_3,a_3),(x_4,a_4)) \\
&\quad -(-1)^{(\bar x_1+\bar x_2)(\bar x_3+\bar x_4)}\rho((x_3,a_3),(x_4,a_4))\rho((x_1,a_1),(x_2,a_2)).
\end{align*}
Therefore, \eqref{eq:mod1} holds. Similarly, we can prove \eqref{eq:mod2}. Then $\rho_1$ is a representation of $E$ over $A$. For all $b\in \mathcal H(A)$, we have
\begin{eqnarray*}
&&[(x_1, a_1), (x_2, a_2), b(x_3, a_3)]_1 = [(x_1, a_1), (x_2, a_2), (bx_3, ba_3)]_1\\
&&=([x_1, x_2, bx_3], \rho(x_1, x_2)(ba_3) -(-1)^{\bar x_2\bar x_3} \rho(x_1, bx_3)a_2 +(-1)^{\bar x_1(\bar x_2+\bar x_3)} \rho(x_2,bx_3)a_1)\\
&&=((-1)^{\bar b(\bar x_1+\bar x_2)}b[x_1, x_2, x_3] + \rho(x_1, x_2)bx_3, \rho(x_1, x_2)(ba_3) -(-1)^{\bar x_2\bar x_3} \rho(x_1, bx_3)a_2 \\
&& \hspace{7,4cm} +(-1)^{\bar x_1(\bar x_2+\bar x_3)} \rho(x_2,bx_3)a_1)\\
&&=(-1)^{\bar b(\bar x_1+\bar x_2)}b[(x_1, a_1), (x_2, a_2), (x_3, a_3)] + \rho_1(x_1, x_2)b(x_3, a_3).
\end{eqnarray*}
Moreover, we have
\begin{align*}
 \rho_1(c(x,a),(y,b))&=\rho_1((cx,ca),(y,b))=\rho(cx,y)\\
&=(-1)^{\bar c\bar x}\rho(x,cy)=(-1)^{\bar c\bar x}\rho_1((x,a), (cy,cb))=(-1)^{\bar c\bar x}\rho_1((x,a), c(y,b)).
\end{align*}
Similarly, we obtain $\rho_1(c(x,a),(y,b))=c\rho_1((x,a),(y,b))$. Then $(E, A, [\cdot,\cdot,\cdot]_1,\rho_1)$ is a $3$-Lie-Rinehart superalgebra.
\end{proof}

\section{Cohomology and deformations of $3$-Lie-Rinehart superalgebras} \label{sec:cohomdef3LieRinesuperalg}
\def\theequation{\arabic{section}. \arabic{equation}}
\setcounter{equation} {0}
In this section, we study the  notion of
a module for a $3$-Lie-Rinehart superalgebras and subsequently we introduce
a cochain complex and cohomology of a $3$-Lie-Rinehart superalgebras with coefficients in a module, then we study relations between 1 and 2 cocycles of a Lie-Rinehart superalgebra
and the induced $3$-Lie-Rinehart superalgebra.
At the end of this section, we study the deformation of $3$-Lie-Rinehart superalgebras

\subsection{Representations and cohomology of $3$-Lie-Rinehart superalgebras }

\begin{defn} \label{def-rep}Let $M$ be an $A$-module. and  $\psi: L\o L\rightarrow End(M)$ be an even linear map. The pair $(M,\psi)$ is called a left module of the $3$-Lie-Rinehart superalgebra  $(L, A, [\cdot,\cdot,\cdot],\rho)$ if the following conditions hold:
\begin{enumerate}[label=\upshape{(\roman*)},left=7pt]
\item  $\psi$ is a representation of $(L, [\c, \c, \c])$ on $M$,
\item $\psi(a\c x, y)=(-1)^{\bar a\bar x}\psi(x, a\c y)=a\c \psi(x, y)$, for all $a\in \mathcal H(A)$ and $x, y\in \mathcal H(L)$,
\item $\psi(x, y)(a\c m)=(-1)^{\bar a(\bar x+\bar y)}a\c \psi(x, y)(m)+\rho(x, y)(a)m$, for all $x, y\in \mathcal H(L)$, $a\in \mathcal H(A)$ and $m\in M$.
\end{enumerate}
\end{defn}

\begin{ex}
$A$ is a left module over $L$ since $\rho$ is a representation of $(L, [\c, \c, \c])$
over $A$ and  the other conditions are satisfied automatically by definition of the map $\rho$.
\end{ex}
\begin{ex}
The pair $(L,ad)$ is a left module over $L$, which is called the adjoint representation of $(L, A, [\cdot,\cdot,\cdot],\rho)$.
\end{ex}

\begin{prop}
Let $(L, A,[\cdot,\cdot,\cdot],\rho)$ be a $3$-Lie-Rinehart superalgebra. Then $(M, \psi)$ is a left module over $(L, A,[\cdot,\cdot,\cdot],\rho)$  if and only if $(L\oplus M, A, [\cdot,\cdot,\cdot]_{L\oplus M},\rho_{L\oplus M})$ is a $3$-Lie-Rinehart superalgebra with the multiplication:
\begin{multline*}
[x_1+m_1,x_2+m_2,x_3+m_3]_{L\oplus M}:=[x_1, x_2, x_3]+\psi(x_1, x_2)m_3+(-1)^{\bar x_1(\bar x_2+\bar x_3)}\psi(x_2, x_3)m_1\\
\hspace{9cm} +(-1)^{\bar x_2\bar x_3}\psi(x_3, x_1)m_2,\\
\rho_{L\oplus M}: (L\oplus M) \otimes (L\oplus M) \rightarrow Der(A), \quad  \rho_{L\oplus M} (x_1+m_1, x_2+m_2):=\rho(x_1, x_2),
\end{multline*}
for any $x_1, x_2, x_3\in \mathcal H(L)$ and $m_1,m_2, m_3\in M$.\\
Note that $(L\oplus M)_{\overline 0}=L_{\overline 0}\oplus M_{\overline 0}$, implying that if $x+m\in \mathcal H(L\oplus M)$, then $\overline{x+m}=\bar{x}=\bar{m}$.

\end{prop}
\begin{proof}  Since $L$ and $M$ are $A$-modules, then $L\oplus M$ is an $A$-module via
\begin{eqnarray*}
a(x+m)=ax+am, \forall a\in A, x\in L, m\in M.
\end{eqnarray*}
 If $(M, \psi)$ is a left module over $(L, A,[\cdot,\cdot,\cdot],\rho)$. Then $(L\oplus M, [\cdot,\cdot,\cdot]_{L\oplus M})$ is a $3$-Lie superalgebra.  It is obvious that $ \rho_{L\oplus M}$ is a representation of the $3$-Lie superalgebra  $(L\oplus M, [\cdot,\cdot,\cdot]_{L\oplus M})$ over $A$.

  For any $x_1, x_2, x_3\in \mathcal H(L)$,  $m_1,m_2, m_3\in \mathcal H(M)$ and $a\in \mathcal H(A)$,
 \begin{align*}
&[x_1+m_1,x_2+m_2,a(x_3+m_3)]_{L\oplus M} =[x_1+m_1,x_2+m_2,ax_3+am_3]_{L\oplus M}\\
=&[x_1, x_2, ax_3]+\psi(x_1, x_2)am_3+(-1)^{\bar x_1(\bar x_2+\bar x_3)}\psi(x_2, ax_3)m_1+(-1)^{\bar x_2\bar x_3}\psi(ax_3, x_1)m_2\\
=&(-1)^{\bar a(\bar x_1+\bar x_2)}a[x_1, x_2, x_3]+\psi(x_1, x_2)am_3+(-1)^{\bar x_1(\bar x_2+\bar x_3)}\psi(x_2, ax_3)m_1\\
& \hspace{2,6cm} +(-1)^{\bar x_2\bar x_3}\psi(ax_3, x_1)m_2+\rho(x_1,x_2)(am_3)\\
=&(-1)^{\bar a(\bar x_1+\bar x_2)}a([x_1, x_2, x_3]+\psi(x_1, x_2)m_3+(-1)^{\bar x_1(\bar x_2+\bar x_3)}\psi(x_2, x_3)m_1 +(-1)^{\bar x_2\bar x_3}\psi(x_3, x_1)m_2)\\
& \hspace{6cm} +\rho(x_1,x_2)(am_3)\\
=& (-1)^{\bar a(\bar x_1+\bar x_2)}a[x_1+m_1,x_2+m_2,x_3+m_3]_{L\oplus M}+\rho_{L\oplus M} (x_1+m_1, x_2+m_2)a(x_3+m_3).
 \end{align*}
Moreover,
\begin{align*}
 \rho_{L\oplus M}(a(x_1+m_1),x_2+m_2)=\rho_{L\oplus M}(ax_1+am_1),x_2+m_2)
 =\rho(ax_1,x_2)=(-1)^{\bar a\bar x_1}\rho(x_1,ax_2)\\
=(-1)^{\overline a~\overline{x_1+m_1}}\rho_{L\oplus M}(x_1+m_1, ax_2+am_2)
=(-1)^{\overline a~\overline{x_1+m_1}}\rho_{L\oplus M}(x_1+m_1, a(x_2+m_2)).
\end{align*}
Similarly, we obtain $\rho_{L\oplus M}(a(x_1+m_1),x_2+m_2)=a\rho_{L\oplus M}(x_1+m_1,x_2+m_2)$.
Therefore, $(L\oplus M, A,  [\cdot,\cdot,\cdot]_{L\oplus M},\rho_{L\oplus M})$ is a $3$-Lie-Rinehart superalgebra.
The sufficient condition can be done in the same way.
\end{proof}
Let  $(M,\psi)$ be a left module of the $3$-Lie-Rinehart superalgebra  $(L, A, [\cdot,\cdot,\cdot],\rho)$ and  $C^{n}(L, M)$ the space of all linear maps   $f: \wedge^2L\o \cdots \o \wedge^2L \wedge L\rightarrow M $ satisfying the conditions below:
\begin{enumerate}
\item $f(x_1, \dots,x_i,x_{i+1}, \dots, x_{2n}, x_{2n+1})=-(-1)^{\bar x_i\bar x_{i+1}}f(x_1,  \dots,x_{i+1},x_i, \dots, x_{2n+1})$,  for  all $x_i\in\mathcal H (L),  1\leq i \leq 2n+1$,
\item $f(x_1,\c \c \c, a\c x_i, \c \c \c, x_{2n+1})=(-1)^{\bar a(\bar x_1+\dots+\bar x_{i-1}+\bar f)}af(x_1,\c \c \c, x_i, \c \c \c, x_{2n+1})$,  for  all $x_i\in \mathcal H(L),  1\leq i \leq 2n+1$  and $a\in \mathcal H(A)$.
\end{enumerate}
Next we consider the $\mathbb{Z}_{+}$-graded space of $\mathbb{K}$-modules
\begin{eqnarray*}
C^{\ast}(L, M):=\oplus_{n\geq 0}C^{n}(L, M).
\end{eqnarray*}

Define the $\mathbb{K}$-linear maps $\delta_{3LR}:  C^{n-1}(L, M)\rightarrow C^{n}(L, M)$ given by
\begin{multline*}
\delta_{3LR} f(x_1,\c  \c  \c, x_{2n+1}) = \\
(-1)^{n+(\bar f+\bar x_1+\bar x_2+\dots +\bar x_{2n-2}) (\bar x_{2n-1}+\bar x_{2n+1})+\bar x_{2n+1}~\bar x_{2n}} \psi(x_{2n-1}, x_{2n+1})f(x_1, \c \c \c, x_{2n-2}, x_{2n})\\
+ (-1)^{n+(\bar f+\bar x_1+\bar x_2+\dots +\bar x_{2n-1}) (\bar x_{2n}+\bar x_{2n+1})} \psi(x_{2n}, x_{2n+1})f(x_1, \c \c \c, x_{2n-1})\\
+\sum^{n}_{k=1}(-1)^{k+(\bar f+\bar x_1+\bar x_2+\dots +\bar x_{2k-2}) (\bar x_{2k-1}+\bar x_{2k})}\psi(x_{2k-1},x_{2k})f(x_1, \c \c \c, \widehat{x}_{2k-1}, \widehat{x}_{2k},\c \c \c, x_{2n+1})\\
+\sum^{n}_{k=1}\sum_{j=2k+1}^{2n+1}(-1)^{k+(\bar x_{2k+1}+\dots +\bar x_{j-1})(\bar x_{2k-1}+\bar x_{2k})}f(x_1, \c \c \c, \widehat{x}_{2k-1}, \widehat{x}_{2k},\c \c \c, [x_{2k-1}, x_{2k}, x_j], \c \c  \c,  x_{2n+1}).
\end{multline*}
\begin{prop}
If $f\in C^{n-1}(L, M)$, then $\delta_{3LR}f\in C^{n}(L, M)$ and $\delta^2_{3LR}=0$.
\end{prop}
\begin{proof}
Let $f$ be a homogenous element in $C^{n-1}(L,M)$, it is obvious that $\delta_{3LR} f$ is skew-symmetric. For all $x_1,x_2,\cdots,x_{2n+1}\in \mathcal H(L)$, $a\in \mathcal H(A)$ and  for $i<2n-1$,
\begin{align*}
&\delta_{3LR} f(x_1,\c \c\c ,ax_i,  \c\c\c, x_{2n+1})=\\
&(-1)^{n+(\bar f+\bar x_1+\bar x_2+\dots +\bar x_{2n-2}) (\bar x_{2n-1}+\bar x_{2n+1})+\bar x_{2n+1}~\bar x_{2n}} \psi(x_{2n-1}, x_{2n+1})f(x_1, \c \c \c,ax_i,  \c\c\c, x_{2n-2}, x_{2n})\\
+& (-1)^{n+(\bar f+\bar x_1+\bar x_2+\dots +\bar x_{2n-1}) (\bar x_{2n}+\bar x_{2n+1})} \psi(x_{2n}, x_{2n+1})f(x_1, \c \c \c,ax_i,  \c\c\c, x_{2n-1})\\
+&\sum^{n}_{\underset{i<2k-1}{k=1}}(-1)^{k+(\bar f+\bar a+\bar x_1+\bar x_2+\dots +\bar x_{2k-2}) (\bar x_{2k-1}+\bar x_{2k})}\psi(x_{2k-1},x_{2k})f(x_1, \c \c \c,ax_i,  \c\c\c, \widehat{x}_{2k-1}, \widehat{x}_{2k},\c \c \c, x_{2n+1})\\
+&\sum^{n}_{\underset{i>2k}{k=1}}(-1)^{k+(\bar f+\bar x_1+\bar x_2+\dots +\bar x_{2k-2}) (\bar x_{2k-1}+\bar x_{2k})}\psi(x_{2k-1},x_{2k})f(x_1,   \c\c\c, \widehat{x}_{2k-1}, \widehat{x}_{2k},\c \c \c,ax_i, \c \c\c, x_{2n+1})\\
+&\sum^{n}_{k=1}(-1)^{k+(\bar f+\bar x_1+\bar x_2+\dots +\bar x_{2k-2}) (\bar x_{2k-1}+\bar a+\bar x_{2k})}\psi(ax_{2k-1},x_{2k})f(x_1, \c \c \c, \widehat{x}_{2k-1}, \widehat{x}_{2k},\c \c \c, x_{2n+1})\\
+&\sum^{n}_{k=1}(-1)^{k+(\bar f+\bar x_1+\bar x_2+\dots +\bar x_{2k-2}) (\bar x_{2k-1}+\bar a+\bar x_{2k})}\psi(x_{2k-1},ax_{2k})f(x_1, \c \c \c, \widehat{x}_{2k-1}, \widehat{x}_{2k},\c \c \c, x_{2n+1})\\
+&\sum^{n}_{\underset{i<2k}{k=1}}\sum_{j=2k+1}^{2n+1}(-1)^{k+(\bar x_{2k+1}+\dots +\bar x_{j-1})(\bar x_{2k-1}+\bar x_{2k})} \\
& \hspace{3cm} f(x_1,\c \c \c,ax_i, \c \c \c, \widehat{x}_{2k-1}, \widehat{x}_{2k},\c \c \c, [x_{2k-1}, x_{2k}, x_j], \c \c  \c,  x_{2n+1})\\
+&\sum^{n}_{k=1}\sum_{\underset{2k-1<i<j}{j=2k+1}}^{2n+1}(-1)^{k+(\bar a+\bar x_{2k+1}+\dots +\bar x_{j-1})(\bar x_{2k-1}+\bar x_{2k})} \\
& \hspace{3cm} f(x_1, \c \c \c, \widehat{x}_{2k-1}, \widehat{x}_{2k},\c \c \c,ax_i, \c \c \c,[x_{2k-1}, x_{2k}, x_j], \c \c  \c,  x_{2n+1})\\
+&\sum^{n}_{k=1}\sum_{\underset{j<i}{j=2k+1}}^{2n+1}(-1)^{k+(\bar x_{2k+1}+\dots +\bar x_{j-1})(\bar x_{2k-1}+\bar x_{2k})} \\
& \hspace{3cm} f(x_1, \c \c \c, \widehat{x}_{2k-1}, \widehat{x}_{2k},\c \c \c, [x_{2k-1}, x_{2k}, x_j], \c \c  \c,ax_i, \c \c \c,  x_{2n+1})\\
+&\sum^{n}_{k=1}\sum_{j=2k+1}^{2n+1}(-1)^{k+(\bar x_{2k+1}+\dots +\bar x_{j-1})(\bar x_{2k-1}+\bar a+\bar x_{2k})}f(x_1, \c \c \c, \widehat{x}_{2k-1}, \widehat{x}_{2k},\c \c \c, [ax_{2k-1}, x_{2k}, x_j], \c \c  \c,  x_{2n+1})\\
+&\sum^{n}_{k=1}\sum_{j=2k+1}^{2n+1}(-1)^{k+(\bar x_{2k+1}+\dots +\bar x_{j-1})(\bar x_{2k-1}+\bar a+\bar x_{2k})}f(x_1, \c \c \c, \widehat{x}_{2k-1}, \widehat{x}_{2k},\c \c \c, [x_{2k-1},ax_{2k}, x_j], \c \c  \c,  x_{2n+1})\\
+&\sum^{n}_{k=1}\sum_{j=2k+1}^{2n+1}(-1)^{k+(\bar x_{2k+1}+\dots +\bar x_{j-1})(\bar x_{2k-1}+\bar x_{2k})}f(x_1, \c \c \c, \widehat{x}_{2k-1}, \widehat{x}_{2k},\c \c \c, [x_{2k-1}, x_{2k},ax_j], \c \c  \c,  x_{2n+1}).\\
\end{align*}
Using Definition \ref{def-rep} and \eqref{eq:Rinhart1}, we obtain
$$
\delta_{3LR} f(x_1,\c \c\c ,ax_i,  \c\c\c, x_{2n+1})=(-1)^{\bar a(\bar x_1+\dots+\bar x_{i-1}+\bar f)}a\delta_{3LR}f(x_1,\c \c\c ,x_i,  \c\c\c, x_{2n+1}),
$$
Similarly, we can proof the same result if $i=2n-1,2n,2n+1$.
Then $\delta_{3LR}$ is well-defined.
Further, $\delta^2_{3LR} = 0$ follows from the direct but a long calculation.
 \end{proof}

By the above proposition, $(C^{\ast}(L,M), \delta_{3LR})$ is a cochain complex. The resulting cohomology of the
cochain complex can be defined as the cohomology space of $3$-Lie-Rinehart superalgebra $(L,A,[\cdot,\cdot,\cdot],\rho)$ with
coefficients in $(M, \psi)$, and we denote this cohomology as $H^{\ast}(L,M)$.
\begin{defn}
Let $(L, A,[\cdot,\cdot,\cdot],\rho)$ be a $3$-Lie-Rinehart superalgebra  and  $(M, \psi)$ be a left module over $L$.  If $\nu \in H^1_{3LR}(L,M)$ satisfies
{\small\begin{eqnarray*}
&&(-1)^{\bar \nu(\bar x_1+\bar x_2)}\psi(x_1, x_2)\nu(x_3)+(-1)^{\bar x_2\bar x_3+\bar\nu(\bar x_2+\bar x_3)}\psi(x_1, x_3)\nu(x_2)\\
&&+(-1)^{(\bar x_1+\bar \nu)(\bar x_2+\bar x_3)}\psi(x_2, x_3)\nu(x_1)+\nu([x_1, x_2, x_3])=0,
\end{eqnarray*}}
for any $x_1, x_2,x_3\in \mathcal H(L)$,  then $\nu$ is called a 1-cocycle associated with $\psi$.
\end{defn}
\begin{defn}
Let $(L, A,[\cdot,\cdot,\cdot],\rho)$ be a $3$-Lie-Rinehart superalgebra  and  $(M, \psi)$ be a left module over $L$.  If $\omega \in H^2_{3LR}(L,M)$  satisfies
\begin{align*}
&(-1)^{(\bar \omega+\bar x_1+\bar x_2)(\bar x_3+\bar x_5)+\bar x_4\bar x_5}\psi(x_3, x_5)\omega(x_1, x_2, x_4)+(-1)^{(\bar \omega+\bar x_1+\bar x_2+\bar x_3)(\bar x_4+\bar x_5)}\psi(x_4, x_5)\omega(x_1, x_2, x_3)\\
& -(-1)^{\bar \omega(\bar x_1+\bar x_2)}\psi(x_1, x_2)\omega(x_3, x_4, x_5)+(-1)^{(\bar \omega+\bar x_1+\bar x_2)(\bar x_3+\bar x_4)}\psi(x_3, x_4)\omega(x_1, x_2, x_5)\\
&-\omega([x_1, x_2, x_3], x_4, x_5)-(-1)^{\bar x_3(\bar x_1+\bar x_2)}\omega(x_3,[x_1, x_2, x_4], x_5)\\
&-(-1)^{(\bar x_3+\bar x_4)(\bar x_1+\bar x_2)}\omega( x_3, x_4, [x_1, x_2, x_5]) +\omega(x_1,x_2,[x_3, x_4, x_5])=0,
\end{align*}
for any $x_1, x_2,x_3,x_4,x_5\in \mathcal H (L)$,  then $\omega $ is called a 2-cocycle associated with $\psi$.
\end{defn}

\begin{thm}\label{AKMS-Z2ad}
Let $(L, A,[\cdot,\cdot],\mu)$ be a Lie-Rinehart superalgebra, $\tau$ be a supertrace  and $\varphi \in Z^2_{LR}(L,L)$ such
that
\begin{align*}
& \forall x,y,z \in \mathcal H(A): \\
& \tau \AKMSpara{x} \tau \AKMSpara{ \varphi \AKMSpara{ y,z} }-(-1)^{\bar x\bar y}\tau \AKMSpara{y} \tau \AKMSpara{ \varphi \AKMSpara{ x,z} }+(-1)^{\bar z(\bar x+\bar y)}\tau \AKMSpara{z} \tau \AKMSpara{ \varphi \AKMSpara{ x,y} } = 0.
\end{align*}
Define the linear map $\phi:\otimes^3L\to L$ by
 $$\phi\AKMSpara{x,y,z}= \tau \AKMSpara{x} \varphi \AKMSpara{ y,z}-(-1)^{\bar x\bar y}\tau \AKMSpara{y}  \varphi \AKMSpara{ x,z}+(-1)^{\bar z(\bar x+\bar y)}\tau \AKMSpara{z} \varphi \AKMSpara{x,y}.$$
 Then $\phi$ is a 2-cocycle of the induced $3$-Lie-Rinehart superalgebra $(L, A,[\cdot,\cdot,\cdot]_\tau,\rho_\tau)$.
\end{thm}
\begin{proof}
Let $\varphi \in Z^2_{LR}(L,L)$, it is obvious that $\phi$  is skew-symmetric and $\bar\phi=\bar\varphi$.  Set $x,y,z\in \mathcal H(L)$ then we have
\begin{align*}
\phi\AKMSpara{ax,y,z}&= \tau \AKMSpara{ax} \varphi \AKMSpara{ y,z}-(-1)^{\bar x\bar y}\tau \AKMSpara{y}  \varphi \AKMSpara{ax,z}+(-1)^{\bar z(\bar x+\bar y)}\tau \AKMSpara{z} \varphi \AKMSpara{ax,y}\\
&=a\tau \AKMSpara{x} \varphi \AKMSpara{ y,z}-(-1)^{\bar x\bar y+\bar a\bar \varphi}\tau \AKMSpara{y} a\varphi \AKMSpara{x,z}+(-1)^{\bar z(\bar x+\bar y)+\bar a\bar \varphi}\tau \AKMSpara{z} a\varphi \AKMSpara{x,y}\big)\\
&=(-1)^{\bar a\bar \varphi}a\big(\tau \AKMSpara{x} \varphi \AKMSpara{ y,z}-(-1)^{\bar x\bar y}\tau \AKMSpara{y}  \varphi \AKMSpara{x,z}+(-1)^{\bar z(\bar x+\bar y)}\tau \AKMSpara{z} \varphi \AKMSpara{x,y}\big)\\
&=(-1)^{\bar a\bar \phi}a\phi\AKMSpara{x,y,z}.
 \end{align*}
 Similarly, $\phi\AKMSpara{x,ay,z}=(-1)^{\bar a(\bar x+\bar \phi)}a\phi\AKMSpara{x,y,z}$ and $\phi\AKMSpara{x,y,az}=(-1)^{\bar a(\bar x+\bar y+\bar \phi)}a\phi\AKMSpara{x,y,z}$.
On the other hand, let $x_1,x_2,y_1,y_2,z\in \mathcal H(L)$. Then
\begin{align*}
& \delta_{3LR} \phi \AKMSpara{x_1,x_2,y_1,y_2,z} = \phi \AKMSpara{x_1,x_2,\AKMSbracket{y_1,y_2,z}_\tau} - \phi \AKMSpara{\AKMSbracket{x_1,x_2,y_1}_\tau,y_2,z} \\
-&(-1)^{\bar y_1(\bar x_1+\bar x_2)} \phi\AKMSpara{y_1,\AKMSbracket{x_1,x_2,y_2}_\tau,z}  -(-1)^{(\bar y_1+\bar y_2)(\bar x_1+\bar x_2)} \phi\AKMSpara{y_1,y_2,\AKMSbracket{x_1,x_2,z}_\tau} \\
+& (-1)^{(\bar x_1+\bar x_2)\bar \phi}\AKMSbracket{x_1,x_2,\phi\AKMSpara{y_1,y_2,z}}_\tau- \AKMSbracket{\phi\AKMSpara{x_1,x_2,y_1},y_2,z}_\tau \\
-&(-1)^{\bar y_1(\bar x_1+\bar x_2+\bar \phi)} \AKMSbracket{y_1,\phi\AKMSpara{x_1,x_2,y_2},z}_\tau -(-1)^{(\bar y_1+\bar y_2)(\bar x_1+\bar x_2+\bar \phi)}  \AKMSbracket{y_1,y_2,\phi\AKMSpara{x_1,x_2,z}}_\tau\\
=&\tau\AKMSpara{y_1}\phi\AKMSpara{x_1,x_2,\AKMSbracket{y_2,z}} -\AKMSpara{-1}^{\bar y_1\bar y_2} \tau\AKMSpara{y_2}\phi\AKMSpara{x_1,x_2,\AKMSbracket{z,y_1}} \\
+&\AKMSpara{-1}^{\bar z(\bar y_1+\bar y_2)} \tau\AKMSpara{z}\phi\AKMSpara{x_1,x_2,\AKMSbracket{y_1,y_2}}  - \tau\AKMSpara{x_1}\phi\AKMSpara{\AKMSbracket{x_2,y_1},y_2,z}\\
+&\AKMSpara{-1}^{\bar x_1\bar x_2} \tau\AKMSpara{x_2}\phi\AKMSpara{\AKMSbracket{x_1,y_1},y_2,z}- \AKMSpara{-1}^{\bar y_1(\bar x_1+\bar x_2)}\tau\AKMSpara{y_1}\phi\AKMSpara{\AKMSbracket{x_1,x_2},y_2,z} \\
-& \AKMSpara{-1}^{\bar y_1(\bar x_1+\bar x_2)} \tau\AKMSpara{x_1}\phi\AKMSpara{y_1,\AKMSbracket{x_2,y_2},z} + \AKMSpara{-1}^{\bar y_1(\bar x_1+\bar x_2)+\bar x_1\bar x_2}\tau\AKMSpara{x_2}\phi\AKMSpara{y_1,\AKMSbracket{x_1,y_2},z}\\
-& \AKMSpara{-1}^{(\bar y_1+\bar y_2)(\bar x_1+\bar x_2)}\tau\AKMSpara{y_2}\phi\AKMSpara{y_1,\AKMSbracket{x_1,x_2},z}  -
\AKMSpara{-1}^{(\bar y_1+\bar y_2)(\bar x_1+\bar x_2)} \tau\AKMSpara{x_1}\phi\AKMSpara{y_1,y_2,\AKMSbracket{x_2,z}}\\
+&\AKMSpara{-1}^{(\bar y_1+\bar y_2)(\bar x_1+\bar x_2)+\bar x_1\bar x_2} \tau\AKMSpara{x_2}\phi\AKMSpara{y_1,y_2,\AKMSbracket{x_1,z}} - \AKMSpara{-1}^{(\bar z+\bar y_1+\bar y_2)(\bar x_1+\bar x_2)}\tau\AKMSpara{z}\phi\AKMSpara{y_1,y_2,\AKMSbracket{x_1,x_2}} \\
+&\AKMSpara{-1}^{\bar\phi(\bar x_1+\bar x_2)} \tau\AKMSpara{x_1}\AKMSbracket{x_2,\phi\AKMSpara{y_1,y_2,z}} -\AKMSpara{-1}^{\bar\phi(\bar x_1+\bar x_2)+\bar x_1\bar x_2} \tau\AKMSpara{x_2}\AKMSbracket{x_1,\phi\AKMSpara{y_1,y_2,z}}\\
+&\AKMSpara{-1}^{(\bar x_1+\bar x_2)(\bar y_1+\bar y_2+\bar z)} \tau\AKMSpara{\phi\AKMSpara{y_1,y_2,z}}\AKMSbracket{x_1,x_2} - \tau\AKMSpara{\phi\AKMSpara{x_1,x_2,y_1}}\AKMSbracket{y_2,z}\\
+&\AKMSpara{-1}^{\bar y_2(\bar x_1+\bar x_2+\bar y_1+\bar \phi)} \tau\AKMSpara{y_2}\AKMSbracket{\phi\AKMSpara{x_1,x_2,y_1},z} - \AKMSpara{-1}^{\bar z(\bar x_1+\bar x_2+\bar y_1+\bar y_2+\bar \phi)} \tau\AKMSpara{z}\AKMSbracket{\phi\AKMSpara{x_1,x_2,y_1},y_2}\\
-&\AKMSpara{-1}^{\bar y_1(\bar x_1+\bar x_2+\bar \phi)} \tau\AKMSpara{y_1}\AKMSbracket{\phi\AKMSpara{x_1,x_2,y_2},z} + \AKMSpara{-1}^{\bar y_1\bar y_2}\tau\AKMSpara{\phi\AKMSpara{x_1,x_2,y_2}}\AKMSbracket{y_1,z}\\
-&\AKMSpara{-1}^{(\bar y_1+\bar z)(\bar x_1+\bar x_2+\bar \phi)+\bar z(\bar y_1+\bar y_2)} \tau\AKMSpara{z}\AKMSbracket{y_1,\phi\AKMSpara{x_1,x_2,y_2}}\\
-& \AKMSpara{-1}^{(\bar y_1+\bar y_2)(\bar x_1+\bar x_2+\bar \phi)}\tau\AKMSpara{y_1}\AKMSbracket{y_2,\phi\AKMSpara{x_1,x_2,z}} \\
+&\AKMSpara{-1}^{(\bar y_1+\bar y_2)(\bar x_1+\bar x_2+\bar \phi)+\bar y_1\bar y_2} \tau\AKMSpara{y_2}\AKMSbracket{y_1,\phi\AKMSpara{x_1,x_2,z}}\\
-& \AKMSpara{-1}^{\bar z(\bar y_1+\bar y_2)} \tau\AKMSpara{\phi\AKMSpara{x_1,x_2,z}}\AKMSbracket{y_1,y_2}\\
=& \tau\AKMSpara{y_1} \Big( \tau\AKMSpara{x_1} \varphi\AKMSpara{x_2,\AKMSbracket{y_2,z}} -(-1)^{\bar x_1\bar x_2} \tau\AKMSpara{x_2} \varphi\AKMSpara{x_1,\AKMSbracket{y_2,z}} \\
+&(-1)^{(\bar y_1+\bar y_2)(\bar x_1+\bar x_2)}  \tau\AKMSpara{y_2}\varphi\AKMSpara{\AKMSbracket{x_1,x_2},z}\Big. - (-1)^{(\bar y_1+\bar z)(\bar x_1+\bar x_2)+\bar z\bar y_2}\tau\AKMSpara{z}\varphi\AKMSpara{\AKMSbracket{x_1,x_2},y_2}\\
-& (-1)^{\bar y_1(\bar x_1+\bar x_2+\bar \phi)}\tau\AKMSpara{x_1}\AKMSbracket{\varphi\AKMSpara{x_2,y_2},z} +(-1)^{\bar y_1(\bar x_1+\bar x_2+\bar \phi)+\bar x_1\bar x_2} \tau\AKMSpara{x_2}\AKMSbracket{\varphi\AKMSpara{x_1,y_2},z} \\
-&(-1)^{\bar y_1(\bar x_1+\bar x_2+\bar \phi)+\bar y_2(\bar x_1+\bar x_2)} \tau\AKMSpara{y_2}\AKMSbracket{\varphi\AKMSpara{x_1,x_2},z} - (-1)^{(\bar y_1+\bar y_2)(\bar x_1+\bar x_2+\bar \phi)}\tau\AKMSpara{x_1}\AKMSbracket{y_2,\varphi\AKMSpara{x_2,z}}\\
+&(-1)^{(\bar y_1+\bar y_2)(\bar x_1+\bar x_2+\bar \phi)} \tau\AKMSpara{x_1}\AKMSbracket{y_2,\varphi\AKMSpara{x_2,z}} - (-1)^{(\bar y_1+\bar y_2)(\bar x_1+\bar x_2+\bar \phi)+\bar z(\bar x_1+\bar x_2)} \tau\AKMSpara{z}\AKMSbracket{y_2,\varphi\AKMSpara{x_1,x_2}} \Big)\\
-&(-1)^{\bar y_1\bar y_2} \tau\AKMSpara{y_2}\Big( \tau\AKMSpara{x_1}\varphi\AKMSpara{x_2,\AKMSbracket{y_1,z}}-(-1)^{\bar x_1\bar x_2} \tau\AKMSpara{x_2}\varphi\AKMSpara{x_1,\AKMSbracket{y_1,z}}\\
+&(-1)^{(\bar y_1+\bar y_2)(\bar x_1+\bar x_2)+\bar y_1\bar y_2} \tau\AKMSpara{y_1}\varphi\AKMSpara{\AKMSbracket{x_1,x_2},z}  +(-1)^{(\bar y_1+\bar y_2+\bar z)(\bar x_1+\bar x_2)+\bar y_2(\bar y_1+\bar z)} \tau\AKMSpara{z}\varphi\AKMSpara{y_1,\AKMSbracket{x_1,x_2}}\\
-&(-1)^{\bar y_2(\bar x_1+\bar x_2+\bar \phi)} \tau\AKMSpara{x_1}\AKMSbracket{\varphi\AKMSpara{x_2,y_1},z} +(-1)^{\bar y_2(\bar x_1+\bar x_2+\bar \phi)+\bar x_1\bar x_2} \tau\AKMSpara{x_2}\AKMSbracket{\varphi\AKMSpara{x_1,y_1},z} \\
-&(-1)^{(\bar y_1+\bar y_2)(\bar x_1+\bar x_2)+\bar y_2\bar \phi} \tau\AKMSpara{y_1}\AKMSbracket{\varphi\AKMSpara{x_1,x_2},z} - (-1)^{(\bar y_1+\bar y_2)(\bar x_1+\bar x_2+\bar \phi)}\tau\AKMSpara{x_1}\AKMSbracket{y_1,\varphi\AKMSpara{x_2,z}}\\
+&(-1)^{(\bar y_1+\bar y_2)(\bar x_1+\bar x_2+\bar \phi)+\bar x_1\bar x_2} \tau\AKMSpara{x_2}\AKMSbracket{y_1,\varphi\AKMSpara{x_1,z}} \\
-& (-1)^{(\bar y_1+\bar y_2+\bar z)(\bar x_1+\bar x_2)+(\bar y_1+\bar y_2)\bar \phi} \tau\AKMSpara{z}\AKMSbracket{y_1,\varphi\AKMSpara{x_1,x_2}} \Big)\\
+&(-1)^{\bar z(\bar y_1+\bar y_2)}\tau\AKMSpara{z}\Big( \tau\AKMSpara{x_1}\varphi\AKMSpara{x_2,\AKMSbracket{y_1,y_2}} -(-1)^{\bar x_1\bar x_2} \tau\AKMSpara{x_2}\varphi\AKMSpara{x_1,\AKMSbracket{y_1,y_2}}\\
-& (-1)^{(\bar z+\bar y_1+\bar y_2)(\bar x_1+\bar x_2+\bar z)}\tau\AKMSpara{y_1}\varphi\AKMSpara{y_2,\AKMSbracket{x_1,x_2}}  +(-1)^{(\bar z+\bar y_1+\bar y_2)(\bar x_1+\bar x_2+\bar z)+\bar y_1\bar y_2} \tau\AKMSpara{y_2}\varphi\AKMSpara{y_1,\AKMSbracket{x_1,x_2}}\\
-&(-1)^{\bar z(\bar x_1+\bar x_2+\bar y_1+\bar y_2+\bar \phi)} \tau\AKMSpara{x_1}\AKMSbracket{\varphi\AKMSpara{x_2,y_1},y_2} +(-1)^{\bar z(\bar x_1+\bar x_2+\bar y_1+\bar y_2+\bar \phi)+\bar x_1\bar x_2} \tau\AKMSpara{x_2}\AKMSbracket{\varphi\AKMSpara{x_1,y_1},y_2}\\
-&(-1)^{\bar z(\bar x_1+\bar x_2+\bar y_1+\bar y_2+\bar \phi)+\bar y_1(\bar x_1+\bar x_2)} \tau\AKMSpara{y_1}\AKMSbracket{\varphi\AKMSpara{x_1,x_2},y_2}\\
+& (-1)^{(\bar y_1+\bar z)(\bar x_1+\bar x_2+\bar \phi)+\bar z(\bar y_1+\bar y_2)}\tau\AKMSpara{x_1}\AKMSbracket{y_1,\varphi\AKMSpara{x_2,y_2}}\\
-&(-1)^{(\bar y_1+\bar z)(\bar x_1+\bar x_2+\bar \phi)+\bar z(\bar y_1+\bar y_2)+\bar x_1+\bar x_2} \tau\AKMSpara{x_2}\AKMSbracket{y_1,\varphi\AKMSpara{x_1,y_2}} \\
+&(-1)^{(\bar y_1+\bar y_2+\bar z)(\bar x_1+\bar x_2+\bar \phi)+\bar z(\bar y_1+\bar y_2)+\bar y_2\bar \phi}\tau\AKMSpara{y_2}\AKMSbracket{y_1,\varphi\AKMSpara{x_1,x_2}} \Big)\\
+& \tau\AKMSpara{x_1} \Big( (-1)^{\bar y_2(\bar x_2+\bar y_1)}\tau\AKMSpara{y_2}\varphi\AKMSpara{\AKMSbracket{x_2,y_1},z} -(-1)^{\bar z(\bar y_1+\bar y_2+\bar x_2)} \tau\AKMSpara{z}\varphi\AKMSpara{\AKMSbracket{x_2,y_1},y_2}\\
-&(-1)^{\bar y_1(\bar x_1+\bar x_2)} \tau\AKMSpara{y_1}\varphi\AKMSpara{\AKMSbracket{x_2,y_2},z} -(-1)^{\bar y_1(\bar x_1+\bar x_2)+\bar z(\bar y_1+\bar y_2+\bar x_2)} \tau\AKMSpara{z}\varphi\AKMSpara{y_1,\AKMSbracket{x_2,y_2}}\\
-&(-1)^{(\bar y_1+\bar y_2)(\bar x_1+\bar x_2)} \tau\AKMSpara{y_1}\varphi\AKMSpara{y_2,\AKMSbracket{x_2,z}} +(-1)^{(\bar y_1+\bar y_2)(\bar x_1+\bar x_2)+\bar y_1\bar y_2} \tau\AKMSpara{x_2}\varphi\AKMSpara{y_1,\AKMSbracket{x_2,z}} \\
+&(-1)^{(\bar x_1+\bar x_2)\bar \phi} \tau\AKMSpara{y_1}\AKMSbracket{x_2,\varphi\AKMSpara{y_2,z}} -(-1)^{(\bar x_1+\bar x_2)\bar \phi+\bar y_1\bar y_2} \tau\AKMSpara{y_2}\AKMSbracket{x_2,\varphi\AKMSpara{y_1,z}}\\
+&(-1)^{(\bar x_1+\bar x_2)\bar \phi+\bar z(\bar y_2+\bar y_2)} \tau\AKMSpara{z}\AKMSbracket{x_2,\varphi\AKMSpara{y_1,y_2}} \Big)\\
+&(-1)^{\bar x_1\bar x_2} \tau\AKMSpara{x_2}\Big( -(-1)^{\bar y_2(\bar x_1+\bar x_2)}\tau\AKMSpara{y_2}\varphi\AKMSpara{\AKMSbracket{x_1,y_1},z} + (-1)^{\bar z(\bar x_1+\bar y_1+\bar y_2)}\tau\AKMSpara{z}\varphi\AKMSpara{\AKMSbracket{x_1,y_1},y_2}\\
+&(-1)^{\bar y_1(\bar x_1+\bar x_2)} \tau\AKMSpara{y_1}\varphi\AKMSpara{\AKMSbracket{x_1,y_1},z}+(-1)^{\bar y_1(\bar x_1+\bar x_2)+\bar z(\bar x_1+\bar y_1+\bar y_2)} \tau\AKMSpara{z}\varphi\AKMSpara{y_1,\AKMSbracket{x_1,y_2}}\\
+&(-1)^{(\bar y_1+\bar y_2)(\bar x_1+\bar x_2)} \tau\AKMSpara{y_1}\varphi\AKMSpara{y_2,\AKMSbracket{x_1,z}} - (-1)^{(\bar y_1+\bar y_2)(\bar x_1+\bar x_2)+\bar y_1\bar y_2} \tau\AKMSpara{y_2}\varphi\AKMSpara{y_1,\AKMSbracket{x_1,z}} \\
-&(-1)^{(\bar x_1+\bar x_2)\bar \phi} \tau\AKMSpara{y_1}\AKMSbracket{x_1,\varphi\AKMSpara{y_1,z}} +(-1)^{(\bar x_1+\bar x_2)\bar \phi+\bar y_1\bar y_2} \tau\AKMSpara{y_2}\AKMSbracket{x_1,\varphi\AKMSpara{y_1,z}}\\
-&(-1)^{(\bar x_1+\bar x_2)\bar \phi+\bar z(\bar y_1+\bar y_2)} \tau\AKMSpara{z}\AKMSbracket{x_1,\varphi\AKMSpara{y_1,y_2}} \Big)\\
+&(-1)^{(\bar x_1+\bar x_2)(\bar y_1+\bar y_2+\bar z)} \Big( \tau\AKMSpara{y_1}\tau\AKMSpara{\varphi\AKMSpara{y_2,z}} -(-1)^{\bar y_1\bar y_2} \tau\AKMSpara{y_2}\tau\AKMSpara{\varphi\AKMSpara{y_1,z}}\\
+&(-1)^{\bar z(\bar y_1+\bar y_2)} \tau\AKMSpara{z}\tau\AKMSpara{\varphi\AKMSpara{y_1,y_2}} \Big) \AKMSbracket{x_1,x_2} \\
-& \Big( \tau\AKMSpara{x_1}\tau\AKMSpara{\varphi\AKMSpara{x_2,y_1}} -(-1)^{\bar x_1\bar x_2} \tau\AKMSpara{x_2}\tau\AKMSpara{\varphi\AKMSpara{x_1,y_1}}\\
+&(-1)^{\bar y_1(\bar x_1+\bar x_2)} \tau\AKMSpara{y_1}\tau\AKMSpara{\varphi\AKMSpara{x_1,x_2}} \Big) \AKMSbracket{y_2,z} \\
+&(-1)^{\bar y_1\bar y_2} \Big( \tau\AKMSpara{x_1}\tau\AKMSpara{\varphi\AKMSpara{x_2,y_2}} -(-1)^{\bar x_1\bar x_2} \tau\AKMSpara{x_2}\tau\AKMSpara{\varphi\AKMSpara{x_1,y_2}}\\
+& (-1)^{\bar y_2(\bar x_1+\bar x_2)}\tau\AKMSpara{y_2}\tau\AKMSpara{\varphi\AKMSpara{x_1,x_2}} \Big) \AKMSbracket{y_1,z} \\
-&(-1)^{\bar z(\bar y_1+\bar y_2)} \Big( \tau\AKMSpara{x_1}\tau\AKMSpara{\varphi\AKMSpara{x_2,z}} -(-1)^{\bar x_1\bar x_2} \tau\AKMSpara{x_2}\tau\AKMSpara{\varphi\AKMSpara{x_1,z}} \\
+&(-1)^{\bar z(\bar x_1+\bar x_2)} \tau\AKMSpara{z}\tau\AKMSpara{\varphi\AKMSpara{x_1,x_2}} \Big) \AKMSbracket{y_1,y_2} \\
=& -  \tau\AKMSpara{y_1}\tau\AKMSpara{x_1}\delta_{LR}\varphi\AKMSpara{z,y_2,x_2} - \tau\AKMSpara{y_1}\tau\AKMSpara{x_2}\delta_{LR}\varphi\AKMSpara{y_2,z,x_1} \\
-&\tau\AKMSpara{y_2}\tau\AKMSpara{x_1}\delta_{LR}\varphi\AKMSpara{y_1,z,x_2} - \tau\AKMSpara{y_2}\tau\AKMSpara{x_2}\delta_{LR}\varphi\AKMSpara{z,y_1,x_1} \\
-&\tau\AKMSpara{z}\tau\AKMSpara{x_1}\delta_{LR}\varphi\AKMSpara{y_2,y_1,x_2} - \tau\AKMSpara{z}\tau\AKMSpara{x_2}\delta_{LR}\varphi\AKMSpara{y_1,y_2,x_1}\\
+&(-1)^{(\bar x_1+\bar x_2)(\bar y_1+\bar y_2+\bar z)}\Big( \tau\AKMSpara{y_1}\tau\AKMSpara{\varphi\AKMSpara{y_2,z}} -(-1)^{\bar y_1\bar y_2} \tau\AKMSpara{y_2}\tau\AKMSpara{\varphi\AKMSpara{y_1,z}}\\
+&(-1)^{\bar z(\bar y_1+\bar y_2)} \tau\AKMSpara{z}\tau\AKMSpara{\varphi\AKMSpara{y_1,y_2}} \Big) \AKMSbracket{x_1,x_2} \\
-& \Big( \tau\AKMSpara{x_1}\tau\AKMSpara{\varphi\AKMSpara{x_2,y_1}} -(-1)^{\bar x_1\bar x_2} \tau\AKMSpara{x_2}\tau\AKMSpara{\varphi\AKMSpara{x_1,y_1}}\\
+&(-1)^{\bar y_1(\bar x_1+\bar x_2)} \tau\AKMSpara{y_1}\tau\AKMSpara{\varphi\AKMSpara{x_1,x_2}} \Big) \AKMSbracket{y_2,z} \\
+&(-1)^{\bar y_1\bar y_2} \Big( \tau\AKMSpara{x_1}\tau\AKMSpara{\varphi\AKMSpara{x_2,y_2}} -(-1)^{\bar x_1\bar x_2} \tau\AKMSpara{x_2}\tau\AKMSpara{\varphi\AKMSpara{x_1,y_2}}\\
+& (-1)^{\bar y_2(\bar x_1+\bar x_2)}\tau\AKMSpara{y_2}\tau\AKMSpara{\varphi\AKMSpara{x_1,x_2}} \Big) \AKMSbracket{y_1,z} \\
-&(-1)^{\bar z(\bar y_1+\bar y_2)} \Big( \tau\AKMSpara{x_1}\tau\AKMSpara{\varphi\AKMSpara{x_2,z}} -(-1)^{\bar x_1\bar x_2} \tau\AKMSpara{x_2}\tau\AKMSpara{\varphi\AKMSpara{x_1,z}} \\
+&(-1)^{\bar z(\bar x_1+\bar x_2)} \tau\AKMSpara{z}\tau\AKMSpara{\varphi\AKMSpara{x_1,x_2}} \Big) \AKMSbracket{y_1,y_2}.
\end{align*}
Since for all $x,y,z \in \mathcal H(L),$
\begin{equation*}
\tau \AKMSpara{x} \tau \AKMSpara{ \varphi \AKMSpara{ y,z} }-(-1)^{\bar x\bar y}\tau \AKMSpara{y} \tau \AKMSpara{ \varphi \AKMSpara{ x,z} }+(-1)^{\bar z(\bar x+\bar y)}\tau \AKMSpara{z} \tau \AKMSpara{ \varphi \AKMSpara{ x,y} } = 0,
\end{equation*}
we get $\delta_{3LR} \phi = 0.$
\end{proof}
\begin{cor}\label{AKMS-Z2tri}
Let $\varphi \in Z^2_{LR}(L,\mathbb{K})$.
Then $$\psi\AKMSpara{x,y,z} = \tau \AKMSpara{x} \varphi \AKMSpara{ y,z}-(-1)^{\bar x\bar y}\tau \AKMSpara{y}  \varphi \AKMSpara{ x,z}+(-1)^{\bar z(\bar x+\bar y)}\tau \AKMSpara{z} \varphi \AKMSpara{x,y}$$ is a $2$-cocycle of the induced $3$-Lie-Rinehart superalgebra.
\end{cor}

\begin{thm}
Every $1$-cocycle for the scalar cohomology of the Lie-Rinehart superalgebra $(L,A,\AKMSbracket{\cdot,\cdot},\mu)$ is a $1$-cocycle for the scalar cohomology of the induced $3$-Lie-Rinehart superalgebra $(L,A,\AKMSbracket{\cdot,\cdot,\cdot}_\tau,\rho_\tau)$.
\end{thm}
\begin{proof}
Let $\omega \in Z^1_{LR}(L,\mathbb{K})$.  Then,
$$\forall x,y \in L: \ \delta_{LR} \omega(x,y) = \omega \AKMSpara{\AKMSbracket{x,y}} = 0,$$
which is equivalent to $\AKMSbracket{L,L} \subset \ker \omega.$ It is obvious to prove that  $\AKMSbracket{L,L,L}_\tau \subset \AKMSbracket{L,L}$ and then $\AKMSbracket{L,L,L}_\tau \subset \ker \omega$, that is
\[ \forall x,y,z \in L:\  \omega\AKMSpara{\AKMSbracket{x,y,z}_\tau}=\delta_{3LR} \omega \AKMSpara{x,y,z} = 0,\]
 which means that $\omega$ is a $1$-cocycle for the scalar cohomology of $\AKMSpara{L, A,\AKMSbracket{\cdot,\cdot,\cdot}_\tau,\rho_\tau}$.
\end{proof}

\begin{lem}\label{AKMScoBtau}
Let $\varphi \in C^1(L,\mathbb{K})$. Then, for all $ x,y,z \in L$,
\begin{equation*}
\delta_{3LR} \varphi\AKMSpara{x,y,z} = \tau \AKMSpara{x} \delta_{LR}\varphi\AKMSpara{y,z}-(-1)^{\bar x\bar y}\tau \AKMSpara{y} \delta_{LR}\varphi\AKMSpara{x,z}+(-1)^{\bar z(\bar x+\bar y)}\tau \AKMSpara{z} \delta_{LR}\varphi\AKMSpara{x,y}.
\end{equation*}
\end{lem}
\begin{proof}
Let $\varphi \in C^1(L,\mathbb{K})$, $x,y,z \in L$. Then
\begin{align*}
& \delta_{3LR} \varphi\AKMSpara{x,y,z}= \varphi\AKMSpara{\AKMSbracket{x,y,z}_\tau}\\
&=\tau \AKMSpara{x} \varphi\AKMSpara{\AKMSbracket{y,z}}-(-1)^{\bar x\bar y}\tau \AKMSpara{y} \varphi\AKMSpara{\AKMSbracket{x,z}}+(-1)^{\bar z(\bar x+\bar y)}\tau \AKMSpara{z} \varphi\AKMSpara{\AKMSbracket{x,y}}\\
&=
\tau \AKMSpara{x} \delta_{LR}\varphi\AKMSpara{y,z}-(-1)^{\bar x\bar y}\tau \AKMSpara{y} \delta_{LR}\varphi\AKMSpara{x,z}+(-1)^{\bar z(\bar x+\bar y)}\tau \AKMSpara{z} \delta_{LR}\varphi\AKMSpara{x,y}.
\end{align*}
completing the proof.
\end{proof}

\begin{prop}\label{AKMSB2triv_eq}
Let $\varphi_1,\varphi_2 \in Z^2_{LR}(L,\mathbb{K})$. If $\varphi_1,\varphi_2$ are in the same cohomology class then $\psi_1,\psi_2$ defined by:
\[ \psi_i\AKMSpara{x,y,z} =\tau \AKMSpara{x}\varphi_i\AKMSpara{y,z}-(-1)^{\bar x\bar y}\tau \AKMSpara{y}\varphi_i\AKMSpara{x,z}+(-1)^{\bar z(\bar x+\bar y)}\tau \AKMSpara{z} \varphi_i\AKMSpara{x,y}, i=1,2 \]
are in the same cohomology class.
\end{prop}
\begin{proof}
If $\varphi_1,\varphi_2 \in Z^2_{LR}(L,\mathbb{K})$ are two cocycles in the same cohomology class, that is
$\varphi_2 - \varphi_1 = \delta_{LR}\nu, \nu \in C^1(L,\mathbb{K}),$
then
\begin{align*}
&\psi_2\AKMSpara{x,y,z} -\psi_1\AKMSpara{x,y,z} =  \tau \AKMSpara{x}\varphi_2\AKMSpara{y,z}-(-1)^{\bar x\bar y}\tau \AKMSpara{y}\varphi_2\AKMSpara{x,z}+(-1)^{\bar z(\bar x+\bar y)}\tau \AKMSpara{z} \varphi_2\AKMSpara{x,y}\\
&\quad \quad \quad \quad \quad \quad \quad \quad \quad -\tau \AKMSpara{x}\varphi_1\AKMSpara{y,z}-(-1)^{\bar x\bar y}\tau \AKMSpara{y}\varphi_1\AKMSpara{x,z}+(-1)^{\bar z(\bar x+\bar y)}\tau \AKMSpara{z} \varphi_1\AKMSpara{x,y}\\
&=\tau \AKMSpara{x}(\varphi_2-\varphi_1)\AKMSpara{y,z}-(-1)^{\bar x\bar y}\tau \AKMSpara{y}(\varphi_2-\varphi_1)\AKMSpara{x,z} +(-1)^{\bar z(\bar x+\bar y)}\tau \AKMSpara{z} (\varphi_2-\varphi_1)\AKMSpara{x,y}\\
&=\tau \AKMSpara{x}\delta_{LR}\nu\AKMSpara{y,z}-(-1)^{\bar x\bar y}\tau \AKMSpara{y}\delta_{LR}\nu\AKMSpara{x,z}+(-1)^{\bar z(\bar x+\bar y)}\tau \AKMSpara{z} \delta_{LR}\nu\AKMSpara{x,y} = \delta_{3LR} \nu\AKMSpara{x,y,z},
\end{align*}
which means that $\psi_1$ and $\psi_2$ are in the same cohomology class.
\end{proof}

\subsection{Deformations of $3$-Lie-Rinehart superalgebras}
\def\theequation{\arabic{section}. \arabic{equation}}
\setcounter{equation} {0}

Let $(L, A,[\cdot,\cdot,\cdot],\rho)$ be a $3$-Lie-Rinehart superalgebra. Denote by $\mathbb{K}[[t]]$ the space of formal power series ring with parameter $t$.

\begin{defn}
A deformation of a $3$-Lie-Rinehart superalgebra $(L, A,[\cdot,\cdot,\cdot],\rho)$ is a $\mathbb{K}[[t]]$-trilinear map
\begin{eqnarray*}
[\cdot,\cdot,\cdot]_t: L[[t]]\times L[[t]]\times L[[t]]\rightarrow L[[t]], \quad [x, y, z]_t=\sum_{i\geq 0}t^im_i(x, y, z),
\end{eqnarray*}
with $m_0=[\cdot,\cdot,\cdot]$ and $m_i$ are even trilinear maps for $i\geq 0$,  satisfying \eqref{3skewsuper},  \eqref{eq:Rinhart2} and \eqref{eq:Rinhart1}.
\end{defn}

Let $[\cdot,\cdot,\cdot]_t$ be a deformation of $[\cdot,\cdot,\cdot]$. Then
 \begin{multline*}
\mathrm{m}_t(x, y, \mathrm{m}_t(z, u, v)) - \mathrm{m}_t(\mathrm{m}_t(x, y, u), v, w)-(-1)^{\bar u(\bar x+\bar y)}\mathrm{m}_t(u, \mathrm{m}_t(x, y, v),w) \\
-(-1)^{(\bar u+\bar v)(\bar x+\bar y)}\mathrm{m}_t(u, v, \mathrm{m}_t(x, y, w))=0.
 \end{multline*}

Comparing the coefficients of $t^n$, $n\geq 0$, we get the following equation:
 \begin{multline*}
\sum_{i, j=0}^n (\mathrm{m}_i(x, y, \mathrm{m}_j(z, u, v)) - \mathrm{m}_i(\mathrm{m}_j(x, y, u), v, w)-(-1)^{\bar u(\bar x+\bar y)}\mathrm{m}_i(u, \mathrm{m}_j(x, y, v),w) \\
-(-1)^{(\bar u+\bar v)(\bar x+\bar y)}\mathrm{m}_i(u, v, \mathrm{m}_j(x, y, w)))=0.
 \end{multline*}

  \begin{defn}
  The 3-cochain $m_1$ is called the infinitesimal of the deformation $[\cdot,\cdot,\cdot]_t$. More
generally, if $m_i = 0$ for $1 \leq i \leq (n-1)$ and $m_n$ is non zero cochain, then $m_n$ is called the
$n$-infinitesimal of the deformation $[\cdot,\cdot,\cdot]_t$.

 \end{defn}
 \begin{defn}
 Two deformations $[\cdot,\cdot,\cdot]_t$ and $[\cdot,\cdot,\cdot]'_t$ are said to be equivalent if there exists a formal
automorphism
 \begin{eqnarray*}
\Phi_t : L[[t]]\rightarrow L[[t]],~~~\Phi_t=id_L+\sum_{i\geq 1}t^i\phi_i,
 \end{eqnarray*}
 where $\phi_i: L\rightarrow L$ is a even $\mathbb{K}$-linear maps such that
 \begin{eqnarray*}
  [x, y, z]'_t=\Phi_t^{-1}[\Phi_t(x), \Phi_t(y), \Phi_t(z)]_t.
 \end{eqnarray*}
 \end{defn}

  \begin{defn}
A deformation is called trivial if it is equivalent to the deformation $m_0 = [\cdot,\cdot,\cdot]$.
 \end{defn}

 \begin{defn}
  A $3$-Lie-Rinehart superalgebra is said to be rigid if  every deformation
of it is trivial.
   \end{defn}
 \begin{thm}
 The cohomology class of the infinitesimal of a deformation $[\cdot,\cdot,\cdot]_t$ is determined
by the equivalence class of $[\cdot,\cdot,\cdot]_t$.
 \end{thm}
\begin{proof}
Let $\Phi_t$ represents an equivalence of deformation given by ${m}_t $ and $ \tilde{{m}_t}$. Then
$$
\tilde{{m}_t}(x,y,z)=\Phi_t^{-1}{m}_t(\Phi_t(x),\Phi_t(y),\Phi_t(z)).
$$
Comparing the coefficients of $t$ from both sides of the above equation we have
\begin{multline*}
\tilde{{m}_1}(x,y,z)+\Phi_1([x,y,z])={m}_1(x,y,z)+[\Phi_1(x),y,z]+(-1)^{\bar x\bar \Phi}[x,\Phi_1(y),z] \\
+(-1)^{(\bar x+\bar y)\bar \Phi}[x,y,\Phi_1(z)],
\end{multline*}
or equivalently, ${m}_1-\tilde{{m}_1}=\delta_{3LR}(\phi_1).$
So, cohomology class of infinitesimal of the deformation is determined by the equivalence class of deformation of ${m}_t$.
\end{proof}

\begin{thm}
A non-trivial deformation of a $3$-Lie-Rinehart superalgebra is equivalent to
a deformation whose $n$-infinitesimal cochain is not a coboundary for some $n\geq 1$.
\end{thm}
\begin{proof} Let $\mathrm{m}_t$ be a deformation of $3$-Lie-Rinehart superalgebra with $n$-infinitesimal $m_n$ for
some $n\geq 1$. Assume that there exists a 3-cochain $\phi\in C^{2}(L, L)$ with $\delta (\phi)=m_n$. Set
$\Phi_t=id_L+\phi t^n$ and $\mathrm{\widetilde{m}}=\Phi_t^{-1}\circ \mathrm{m}_t\circ\Phi_t$.
Comparing the coefficients of $t^n$, we get the following equation:
\begin{eqnarray*}
\widetilde{m}_n-m_n=-\delta_{3LR}(\phi).
\end{eqnarray*}
So $\widetilde{m}_n=0$.  Then, one has the deformation whose $n$-infinitesimal is not a coboundary for some $n\geq 1$.
\end{proof}

\end{document}